\documentclass[11pt]{amsart}
\usepackage[english]{babel}
\usepackage[utf8]{inputenc}
\usepackage[T1]{fontenc}
\usepackage[nomath]{lmodern}

\usepackage{amssymb,amsmath,amsthm,amsfonts}
\usepackage{thmtools,mathtools}
\usepackage{microtype}
\allowdisplaybreaks
\usepackage{float}

\usepackage[dvipsnames]{xcolor}
\usepackage{tikz,tikz-cd}
\usepackage{fmtcount}
\usepackage[colorlinks=true,citecolor=blue,urlcolor=blue,linkcolor=blue,breaklinks,pdfencoding=auto,psdextra,pagebackref]{hyperref}
\usepackage[margin=1in]{geometry}
\usepackage[normalem]{ulem}
\usepackage{nicematrix}

\addto\extrasenglish{}
\addto\extrasenglish{}
\addto\extrasenglish{\def\equationautorefname~#1\null{\textnormal{(#1)}\null}}
\declaretheorem[name=Theorem,refname={Theorem},style=plain,numberwithin=section]{theorem}
\declaretheorem[name=Proposition,refname={Proposition},style=plain,sibling=theorem]{proposition}
\declaretheorem[name=Lemma,refname={Lemma},style=plain,sibling=theorem]{lemma}
\declaretheorem[name=Conjecture,refname={Conjecture},style=plain,sibling=theorem]{conjecture}
\declaretheorem[name=Definition,refname={Definition},style=definition,sibling=theorem]{definition}
\declaretheorem[name=Remark,refname={Remark},style=definition,sibling=theorem]{remark}
\declaretheorem[name=Example,refname={Example},style=definition,sibling=theorem]{example}
\declaretheorem[name=Corollary,refname={Corollary},style=plain,sibling=theorem]{corollary}

\declaretheorem[name=Theorem,refname={Theorem},style=plain,numberwithin=section]{theoremLetter}

\declaretheorem[name=Corollary,refname={Corollary},style=plain,sibling=theoremLetter]{corollaryLetter}

\declaretheorem[name=Problem,refname={Problem},style=plain,sibling=theoremLetter]{problemLetter}

\newcounter{chr}
\loop\ifnum\value{chr}<27
  \expandafter\edef\csname b\Alph{chr}\endcsname{\noexpand\mathbf{\Alph{chr}}}
  \expandafter\edef\csname c\Alph{chr}\endcsname{\noexpand\mathcal{\Alph{chr}}}
  \expandafter\edef\csname f\Alph{chr}\endcsname{\noexpand\mathfrak{\Alph{chr}}}
\addtocounter{chr}{1}
\repeat

\def\m{\mathfrak m}
\def\p{\mathfrak p}

\def\tS{{\tilde S}}

\def\Aut{\operatorname{Aut}}
\def\Bl{\operatorname{Bl}}

\def\disc{\operatorname{disc}}

\def\div{\operatorname{div}}

\def\Gr{\operatorname{Gr}}
\def\Hilb{\operatorname{Hilb}}
\def\Hom{\operatorname{Hom}}

\def\Id{\operatorname{Id}}

\def\NS{\operatorname{NS}}

\def\OGr{\operatorname{OGr}}

\def\Pic{\operatorname{Pic}}

\def\rk{\operatorname{rk}}
\def\Sing{\operatorname{Sing}}
\def\SL{\operatorname{SL}}

\def\Spec{\operatorname{Spec}}
\def\Stab{\operatorname{Stab}}
\def\Supp{\operatorname{Supp}}
\def\Sym{\operatorname{Sym}}

\def\KKK{{\mathrm{K3}}}
\def\Kum{{\mathrm{Kum}}}

\let\ordexists\exists
\def\exists{\operatorname{\ordexists}}
\let\ordforall\forall
\def\forall{\operatorname{\ordforall}}

\def\set#1{{\left\{{#1}\right\}}}
\def\setmid#1#2{{\left\{{#1}\;\middle|\;{#2}\right\}}}

\def\tilde{\widetilde}

\def\setminus{\smallsetminus}

\def\git{/\!\!/}

\def\bw#1{{\mathchoice%
 {\textstyle{\bigwedge\mkern-4.5mu^{#1}\mkern1mu}}%
 {\textstyle{\bigwedge\mkern-4.5mu^{#1}\mkern1mu}}%
 {\scriptstyle{\bigwedge\mkern-5mu^{#1}}}%
 {\scriptscriptstyle{\bigwedge\mkern-5mu^{#1}}}%
}}

\def\longarrow#1#2{\mathchoice{#2}{#1}{#1}{#1}}
\def\to{\longarrow{\rightarrow}{\longrightarrow}}
\def\simto{\longarrow{\xrightarrow\sim}{\stackrel\sim\longrightarrow}}
\def\into{\longarrow{\hookrightarrow}{\lhook\joinrel\longrightarrow}}
\def\onto{\longarrow{\twoheadrightarrow}{\relbar\joinrel\twoheadrightarrow}}
\let\shortmapsto\mapsto
\def\mapsto{\longarrow{\shortmapsto}{\longmapsto}}

\def\Kum{{\mathrm{Kum}}}
\def\OG{{\mathrm{OG}}}

\def\SO{{\mathrm{SO}}}

\def\sm{{\mathrm{sm}}}
\def\CY{\mathrm{CY}}

\newcommand{\Syz}{\mathrm{Syz}}
\newcommand{\Pf}{\mathrm{Pf}}
\newcommand{\pf}{\mathrm{pf}}
\newcommand{\red}{\mathrm{red}}
\newcommand\Mcoble[1]{\cM_{\mathrm{Coble#1}}}
\newcommand{\coble}{\mathfrak{c}}

\begin{document}
\title{Coble type hypersurfaces and hyperkähler fourfolds}
\author[B.~Piroddi, Á.D.~Ríos~Ortiz, A.~Rojas, and J.~Song]{Benedetta Piroddi, Ángel David Ríos Ortiz, Andrés Rojas, and Jieao Song}
\date{\today}

\begin{abstract}
A classical result, already observed by Coble, asserts that a genus 2 Jacobian can be embedded in $\bP^8$ as the singular locus of a unique cubic hypersurface; similarly, the Kummer of a genus 3 Jacobian is embedded in $\bP^7$ as the singular locus of a unique quartic hypersurface.
We present a precise analogue of these results in the context of hyperkähler fourfolds:
for the general member in the $20$-dimensional locally complete families of polarized hyperkähler fourfolds of $\KKK^{[2]}$-type with squares $4$ or $6$ and divisibility~$1$,
we establish the existence of a unique \emph{Coble type hypersurface}.
As a consequence, we describe several geometric aspects of the corresponding moduli spaces.
\end{abstract}

\address{Benedetta Piroddi:
Fachrichtung Mathematik, Universität des Saarlandes \hfill \newline\texttt{}
\indent Universität Campus, Gebäude E2.4, 66123 Saarbrücken, Germany}
\email{{\tt piroddi@math.uni-sb.de}}

\address{Ángel David Ríos Ortiz: Université Paris Cité and Sorbonne Université, CNRS, IMJ-PRG \hfill\newline\texttt{}
\indent F-75013 Paris, France}
\email{{\tt riosortiz@imj-prg.fr}}

\address{Andrés Rojas: Departament de Matemàtiques i Informàtica, Universitat de Barcelona \hfill \newline\texttt{}
\indent Gran Via de les Corts Catalanes 585,
08007 Barcelona, Spain}
\email{{\tt andresrojas@ub.edu}}

\address{Jieao Song: Dipartimento di Matematica
``Federigo Enriques'',
Università degli Studi di Milano \hfill \newline\texttt{}
\indent Via Cesare Saldini 50, 20133 Milano, Italy}
\email{{\tt jieao.song@unimi.it}}

\maketitle

\section{Introduction}
A beautiful result on the projective geometry of abelian varieties, already observed by Coble in the early 20th century \cite{Coble:cremona, Coble:theta} (see also \cite{Barth,Laszlo,Beauville}), states that for a genus $2$ curve $C$, the complete linear system $|3\Theta|$ embeds its Jacobian $JC$ as the singular locus of a unique cubic hypersurface in $\bP^8$. Similarly, for a non-hyperelliptic genus $3$ curve $C$, the linear system $|2\Theta|$ on $JC$ embeds the Kummer variety $JC/\langle \pm1\rangle$ as the singular locus of a unique quartic hypersurface in $\bP^7$.
These \emph{Coble hypersurfaces} have far-reaching applications to abelian varieties and their moduli spaces, theta group and theta representations, and the geometry of moduli spaces of vector bundles on curves \cite{van-der-Geer,Narasimhan-Ramanan,Ortega,Nguyen,Gruson-Sam-Weyman,Gruson-Sam,Benedetti-Bolognesi-Faenzi-Manivel:Coble-quadric}.

The goal of this paper is to present a natural analogue of Coble's results in the context of hyperkähler fourfolds of K3$^{[2]}$-type, namely, those that are deformation equivalent to the Hilbert square of a K3 surface. 
Following \cite[Section~3.5]{Debarre:survey}, let us denote by $\cM^{(\gamma)}_{2d}$ the moduli space of polarized hyperkähler fourfolds $(X,H)$ of K3$^{[2]}$-type such that $q(H) = 2d$ and $\div(H) = \gamma$ with respect to the Beauville--Bogomolov--Fujiki quadratic form $q$ on $H^2(X,\bZ)$.
We recall that $\cM^{(\gamma)}_{2d}$ is an irreducible quasiprojective variety of dimension $20$ for $\gamma=1$ and all $2d\ge 2$, or $\gamma=2$ and $2d \equiv 6\pmod 8$. Moreover, the polarization of a general element in $\cM^{(\gamma)}_{2d}$ is very ample whenever $2d\geq 4$ (see \cite[Corollary~B]{projective-models} for $2d=4$ and \cite[Corollary~3.9]{Debarre:survey} for all other cases).

Our main results concern the two families $\cM_4^{(1)}$ and $\cM_6^{(1)}$ and read as follows.

\begin{theoremLetter}
\label{thm:coble-quartic}
For a general $(X,H)\in \cM_4^{(1)}$ and the corresponding embedding 
\[
X\into \bP^{9}=\bP H^0(X, H)^\vee,
\]
there exists a unique quartic hypersurface $Y\subset \bP^{9}$ such that, scheme theoretically, $X$ coincides with the singular locus of $Y$.
\end{theoremLetter}

\begin{theoremLetter}
\label{thm:coble-cubic}
For a general $(X,H)\in \cM_6^{(1)}$ and the corresponding embedding 
\[
X\into \bP^{14}=\bP H^0(X, H)^\vee,
\]
there exists a unique cubic hypersurface $Y\subset \bP^{14}$ such that, scheme theoretically, $X$ coincides with the singular locus of $Y$.
\end{theoremLetter}

In view of their striking resemblance to the classical Coble hypersurfaces, we will refer to these hypersurfaces as \emph{Coble type hypersurfaces}.

As a first consequence, we can also determine the (saturated) homogeneous ideal in the projective space for a general $(X,H)$ of these two families (see \cite[Theorem~E.(1)]{projective-models} and \autoref{cor:square-6-homogeneous-ideal}).
\begin{corollaryLetter}
The homogeneous ideal of a general $(X,H)\in\cM_4^{(1)}$ is generated by the $10$ polar cubics of $Y$ and $20$ more quartics.

Similarly, the homogeneous ideal of a general $(X,H)\in \cM_6^{(1)}$ is generated by the $15$ polar quadrics of $Y$ and $20$ more cubics.
\end{corollaryLetter}

The number $20$ appears in both cases, and is equal to the dimension of the two moduli spaces; the same phenomenon can be observed for the two classical Coble examples as well. We will give an explanation for this remarkable fact in \autoref{subsec:coble-deformation}, relating the homogeneous ideal of $X$ to its deformation space.

\bigskip

To put our results into perspective, let us mention that there are only a few explicit descriptions of locally complete families of higher-dimensional projective hyperkähler varieties, and constructing new ones proves to be an extremely challenging problem.
For $\KKK^{[2]}$-type, the families described in the literature are double EPW sextics $\cM_2^{(1)}$ \cite{OGrady:EPW}, varieties of lines of cubic fourfolds $\cM_6^{(2)}$ \cite{Beauville-Donagi}, Debarre--Voisin fourfolds $\cM_{22}^{(2)}$ \cite{Debarre-Voisin}, and varieties of sums of powers $\cM_{38}^{(2)}$ \cite{Iliev-Ranestad};
in higher dimensions,
$\cM_{\KKK^{[4]},2}^{(2)}$ \cite{Lehn-Lehn-Sorger-vanStraten} and $\cM_{\KKK^{[3]},4}^{(2)}$ \cite{Iliev-Kapustka-Kapustka-Ranestad:EPW-cubes} (see also \cite{stability-families,Perry-Pertusi-Zhao}).
Our work represents a major step towards understanding the two families $\cM_4^{(1)}$ and $\cM_6^{(1)}$, which are arguably the most natural open cases due to their low square, following the work initiated in \cite{projective-models} for special members (actual Hilbert squares) in these two families.
We also remark that, since the existence of these Coble type hypersurfaces is rather incidental, one cannot hope to extend this construction to other types of polarizations (see \autoref{prop:HK-coble}).

For the family $\cM_4^{(1)}$ we will discuss in \autoref{sec:moduli_spaces} the construction of a GIT moduli space $\Mcoble4$ parametrizing Coble type quartics in $\bP^9$, equipped with a birational ``modular'' map
\[
\m\colon \Mcoble4 \dasharrow \cM_4^{(1)}.
\]
A similar construction can be carried out for a moduli space $\Mcoble3$ of Coble type cubics in $\bP^{14}$.
While an explicit parametrization for $\Mcoble4$ and $\Mcoble3$ is still missing, we manage to formulate some rather precise speculations on the global geometry of $\Mcoble4$ in \autoref{sec:speculations}. More precisely, we consider the composition of $\m$ with the period map
\[
\p\colon\cM_4^{(1)} \into \cP_4^{(1)}
\]
(which is an open immersion by the global Torelli theorem) and take the Baily--Borel compactification $\overline{\cP_4^{(1)}}$ for the period domain. The notion of \emph{Hassett--Looijenga--Shah (HLS) divisor} (see \autoref{def:HLS}) yields a natural way of comparing the two compactifications $\Mcoble4$ and $\overline{\cP_4^{(1)}}$. We provide speculations on HLS divisors based on analogy with the known cases of cubic fourfolds \cite{Voisin:torelli,Hassett,Looijenga,Laza:moduli-via-period}, double EPW sextics \cite{OGrady:taxonomy,OGrady:moduli-EPW}, and Debarre--Voisin fourfolds \cite{Debarre-Han-OGrady-Voisin,Oberdieck,Song:special-DV}, and we include an explicit geometrical analysis of certain special loci. One such locus in particular, corresponding to Hilbert squares of K3 surfaces of genus $7$, turns out to be crucial in the proof of \autoref{thm:coble-quartic}.

One is then naturally led to consider the following two problems.
\begin{problemLetter}
Find explicit (unirational) parametrizations for the two moduli spaces of Coble type hypersurfaces $\Mcoble4$ and $\Mcoble3$.
\end{problemLetter}
\begin{problemLetter}
Investigate the relations between the hyperkähler fourfold $X$ and its Coble type hypersurface $Y$, at the level of Hodge structures, motives, and derived categories.
\end{problemLetter}

Finally, we also mention \cite{Gruson-Sam-Weyman} and the recent works of Benedetti--Bolognesi--Faenzi--Manivel \cite{Benedetti-Bolognesi-Faenzi-Manivel:Coble-quadric, Benedetti-Bolognesi-Faenzi-Manivel:Hecke, Benedetti-Bolognesi-Faenzi-Manivel:Gopel}, where the classical examples of Coble hypersurfaces are being studied under the lens of Vinberg theory and generalized to examples in other homogeneous spaces.
It would be interesting to explore if the hyperkähler examples can also be understood from the perspective of Lie theory.

\bigskip

\subsection*{Structure of the paper}
In \autoref{sec:Coble_hypersurfaces_general} we introduce Coble type hypersurfaces, establish some basic properties, and give a selection of examples from the literature, focusing on the case of varieties with trivial canonical bundle.

\autoref{sec:thmA} and \autoref{sec:thmB} are devoted to the proofs of \autoref{thm:coble-quartic} and \autoref{thm:coble-cubic}, respectively.
As the two proofs share a common strategy, we also provide here an outline for the reader's convenience.

\textbf{Step 1.} Show that there exists a unique hypersurface $Y$ of degree $d$ ($d=4$ in \autoref{thm:coble-quartic}, $d=3$ in \autoref{thm:coble-cubic}) whose singular locus contains the hyperkähler fourfold $X$. In other words, compute that $h^0(\bP,\cI^2_{X}(d))=1$.
    \begin{itemize}
        \item First we specialize to the case of polarized Hilbert squares $(S^{[2]},L_2-2\delta)$, where $(S,L)$ is a general K3 surface of genus~$7$ or~$8$. In that case, $S^{[2]}$ can be realized as a degeneracy locus inside some homogeneous ambient variety (a smooth quadric $Q$ for genus $7$, a Grassmannian $G=\Gr(2,6)$ for genus $8$).
        Then one computes $h^0(\bP, \cI^2_{S^{[2]}/\bP}(d))=1$ and $h^1(\bP, \cI^2_{S^{[2]}/\bP}(d))=0$ (\autoref{prep-lemma} and \autoref{prep-lemma-cubic} for the two cases, respectively).
        
        We point out that in the case of genus $8$, an important technical result (\autoref{54again}) on first-order deformations of $S^{[2]}$ in $G$ must first be established; its analogue in the case of genus $7$ was already proven in \cite{projective-models}.
        \item To deduce the general case of Step 1, upper semicontinuity shows that $h^1(\bP, \cI^2_{X/\bP}(d))=0$.
        Then one can conclude using the $d$-normality of $X$ in $\bP$ together with the short exact sequence $0\to \cI_{X/\bP}^2(d)\to \cI_{X/\bP}(d)\to \cN_{X/\bP}^\vee(d)\to 0$.
    \end{itemize}
    
\textbf{Step 2.} Show that the scheme-theoretical singular locus $\Sing(Y)$ coincides with $X$.
    \begin{itemize}
        \item We consider a general one-parameter family $\cX\subset \bP^9\times\Delta\to\Delta$ with central fiber $\cX_0=(S^{[2]}, L_2-2\delta)$, and the corresponding family of hypersurfaces $\cY=Z(F)$ of degree~$d$ such that $\Sing(\cY_t)\supset \cX_t$.
        The relative singular locus $\tilde{\cZ}=Z(\tfrac{\partial F}{\partial x_i})$ will contain~$\cX$ but also a vertical component over $0\in\Delta$ (the quadric $Q$ or the Grassmannian $G$).
        \item Removing this vertical component, one gets a flat family $\cZ$ whose general fiber $\cZ_t$ will still be $\Sing(\cY_t)$.
        Then we show that $\cZ_0$ coincides with $\cX_0=S^{[2]}$ by producing equations that vanish along $\cZ$ and cut out $\cX_0$ scheme-theoretically over $0\in\Delta$.
        
        This step is straightforward in the case of square $4$ but more difficult for square~$6$, due to the large codimension of the Grassmannian $G$ in $\bP^{14}$.
        \autoref{subsec:smoothness} is dedicated to circumvent this difficulty.
        \item Conclude that the general fiber $\cZ_t=\Sing(\cY_t)$ must coincide with $\cX_t$ as well, using upper semicontinuity.
    \end{itemize}

In \autoref{sec:moduli_spaces} we introduce the GIT quotient moduli space $\Mcoble4$ for Coble type quartics in $\bP^9$, and investigate the birational map from $\Mcoble4$ to $\cM_4^{(1)}$ and $\cP_4^{(1)}$.
In particular, we prove that the Heegner divisors $\cD_4$ and $\cD_{12}$ (the former lies on the boundary of $\cM_4^{(1)}$) are the inverse image of certain points in $\Mcoble4$.
We conclude the paper with \autoref{sec:speculations}, where we conjecture in more details about the geometry of the moduli spaces $\Mcoble4$ and $\cM_4^{(1)}$.

\subsection*{Acknowledgments}
We are grateful to Vladimiro Benedetti, Chiara Camere, Ciro Ciliberto, Olivier Debarre, Alice Garbagnati, Franco Giovenzana, Akihiro Kanemitsu, Grzegorz and Michał Kapustka, Martí Lahoz, Ziqi Liu, Laurent Manivel, and Claire Voisin for some useful comments and discussions.

Many geometrical intuitions behind the results in this project would not be possible without extensive experimenting with the \emph{free and open-source} computer algebra systems {\sc Macaulay2} \cite{macaulay2}, {\sc Singular} \cite{singular}, and {\sc Sage} \cite{sagemath}.
We would like to thank the developers for their efforts.

Rojas was supported by the Spanish MICINN project PID2023-147642NB-I00.
Song was partially supported by the PRIN2022 research grant 2022PEKYBJ ``Symplectic varieties: their interplay with Fano manifolds and derived categories''.

\section{Coble type hypersurfaces}
\label{sec:Coble_hypersurfaces_general}

Let $Y\subset \bP^N$ be a hypersurface in $N$-dimensional projective space. We will consider its singular locus $\Sing(Y)$, equipped with the natural scheme structure defined by the intersection of the polar hypersurfaces of $Y$. In this way, the nodal plane cubic $y^2z=x^2(x+z)$ has singular locus consisting of a reduced point, whereas the singular locus of the cuspidal plane cubic $y^2z=x^3$ is a point with non-reduced structure. 

If $Y$ is singular at a (closed) point $p\in \bP^N$, then the embedded tangent space to $\Sing(Y)$ at $p$ is the singular locus of the quadratic tangent cone%
\footnote{The quadratic tangent cone is also called the \emph{polar quadric} of $Y$ with respect to $p$, see \cite[Section~1.1.3]{Dolgachev:classical}. It might be the entire projective space $\bP^N$ (this is the case if and only if $\mathrm{mult}_p(Y)\geq3$).}
to $Y$ at $p$.
For example:
\begin{itemize}
    \item If $\mathrm{mult}_p(Y)\geq3$, then $\Sing(Y)$ has dimension $N$ at $p$.
    \item If $\Sing(Y)$ is smooth at $p$, then $\mathrm{mult}_p(Y)=2$ and the quadratic tangent cone is singular of corank $\dim_p\Sing(Y)+1$. Moreover, the quadratic tangent cone defines a smooth quadric in the (projectivized) normal space $\bP\cN_{\Sing Y/\bP^N}(p)$.
\end{itemize}

\subsection{First properties}

\begin{definition}
Let $X\subset \bP^{N}$ be a projective scheme. We say that $X$ \emph{admits a Coble type hypersurface of degree $d$}, if there exists a hypersurface $Y=Z(f)$ of degree $d$ such that $X$ is scheme-theoretically the singular locus of $Y$.
\end{definition}
The above definition is formulated in the most general way, almost tautological: namely, $X$ can be the singular locus of any hypersurface, and therefore can have non-reduced structure or be non-equidimensional.
However, we will be mostly interested in the following setting:
\begin{equation}\label{eq:setting}
\parbox{\textwidth-\leftmargin}{$X$ is a integral projective variety of dimension $n$, that is regular in codimension $1$ and embedded in $\bP^N=\bP^{n+c}$ via a regular embedding of codimension $c$.
}\tag{$\star$}
\end{equation}

We note that if condition \eqref{eq:setting} is satisfied, the normal sheaf $\cN_{X/\bP^N}$ is locally-free of rank $c$, and $X$ is Gorenstein (hence Cohen--Macaulay; see \cite[Theorem~21.3]{Matsumura}) and normal (by Serre's criterion).
If $X$ is smooth, or more generally, if $X$ only has locally complete intersection singularity in codimension $\ge 2$, then the condition \eqref{eq:setting} is automatically satisfied (see \cite[Theorem~21.2]{Matsumura}).

\begin{proposition}
\label{prop:coble-quadratic-form}
Let $X\subset \bP^N=\bP^{n+c}$ be a variety of dimension $n$ satisfying the condition~\eqref{eq:setting}.
If $X$ admits a Coble type hypersurface $Y$ of degree $d$ in $\bP^N$, then there exists a nowhere degenerate quadratic form
\begin{equation}\label{eq:coble-quadratic-form}
q\colon \Sym^2\cN_{X/\bP^N}\to \cO_X(d).
\end{equation}
\end{proposition}
\begin{proof}
Assume that $X$ admits a Coble type hypersurface $Y$ of degree $d$ in $\bP^N=\bP^{n+c}$.
Along the smooth locus $X^\sm$ of $X$, the tangent cones are smooth quadrics in the projectivization of the normal bundle $\cN_{X/\bP^N}$.
More precisely, if we blow up $X$ in $\bP^N$, we get the following diagram with cartesian squares
\begin{equation}
\label{eq:blowup}
\begin{tikzcd}
E_Y\ar[d,hookrightarrow]\ar[r,hookrightarrow] & E\ar[d,hookrightarrow]\ar[r] & X\ar[d,hookrightarrow]\\
\Bl_X Y \ar[r,hookrightarrow]                 & \Bl_X \bP^N\ar[r]                & \bP^N
\end{tikzcd}
\end{equation}
where $E_Y$ and $E$ are the exceptional divisors, and $E$ is the projective bundle $\bP_X\cN_{X/\bP^N}$.

We see that $\Bl_X Y$ has class $dH-2E$ on $\Bl_X \bP^N$ (for $H$ the hyperplane class in $\bP^N$), and so $E_Y$ provides a generically smooth quadric bundle over $X$ that is smooth over $X^\sm$, which defines the quadratic form.
Since $X$ is regular in codimension $1$, the discriminant $\disc(q)$ of $q$ must be a nonzero section of the trivial line bundle, hence $q$ is nowhere degenerate.
\end{proof}

As an immediate consequence, we obtain the following condition on the existence of a Coble type hypersurface for $X$ in $\bP^N$ in terms of its basic invariants (this was also obtained in \cite[Section~3]{Aluffi}):

\begin{corollary}
\label{cor:dim-codim-condition}
Let $X\subset \bP^N=\bP^{n+c}$ be a variety of dimension $n$ satisfying the condition~\eqref{eq:setting} and admitting a Coble type hypersurface $Y$ of degree $d$ in $\bP^N$.
Then we have
\[
2K_X=\big(d(N-n)-2(N+1)\big)H,
\]
where $H$ is the hyperplane class on $X$. In particular, if $X$ satisfies $\omega_X\simeq \cO_X$,
then $X$ can only admit a Coble type hypersurface of degree $d$ in $\bP^{n + \frac{2(n+1)}{d-2}}$.
Notably, such $X$ can only admit Coble type quartics in $\bP^{2n+1}$ and Coble type cubics in $\bP^{3n+2}$.
\end{corollary}
\begin{proof}
The discriminant $\disc(q)$ of the quadratic form in \eqref{eq:coble-quadratic-form} is a section of the line bundle
\[
(\det \cN_{X/\bP^N}^\vee)^{\otimes 2}\otimes \cO_X(d)^{\otimes c},
\]
where $c=\rk \cN_{X/\bP^N}$ is the codimension of $X$ in $\bP^N$ (see \cite[Section~1.1]{Auel-Bernardara-Bolognesi} or \cite[Theorem~10]{Harris-Tu}).
Since $q$ is nowhere degenerate, this is the trivial line bundle, hence using
\[
\det \cN^\vee_{X/\bP^N}\otimes \omega_X=\omega_{\bP^N}|_X
\]
we obtain
\[
c\cdot d H = (N-n)\cdot d H = 2\big((N+1)H+K_X\big),
\]
which gives the desired equality.
If we now assume that $\omega_X\simeq \cO_X$, then $K_X=0$ and we obtain
$d(N-n)-2(N+1)=0$, or equivalently, $N = n + \tfrac{2(n+1)}{d-2}$.
This concludes the proof.
\end{proof}

\begin{remark}
Note that both the quadratic form in \autoref{prop:coble-quadratic-form} and the condition in \autoref{cor:dim-codim-condition} are local on $X$, hence they can be easily generalized to the case where $X$ is regularly embedded in a smooth projective variety $G$, such that $X$ is the singular locus of a hypersurface $Y\subset G$, or even more generally, a connected component of $\Sing(Y)$ (see \cite{Aluffi}).
\end{remark}

\subsection{Examples}

We are now ready to recover some examples of Coble type hypersurfaces from the literature, and to introduce in this context our main results.
Our focus is principally on higher-dimensional varieties with trivial canonical bundle.
We also refer to \cite[Section~3]{Aluffi} for some other examples.

\begin{example}
As a first trivial example, for a hypersurface $X$ in $\bP^N$ of degree $e$ defined by a polynomial $f$, the non-reduced hypersurface $Y$ of degree $d\coloneqq 2e$ defined by $f^2$ is a Coble type hypersurface for $X$.
One particular case of interest to us is when $X$ is a smooth quadric $Z(q)$ and $Y$ is a non-reduced quartic $Z(q^2)$.
\end{example}
\begin{example}
\label{example:fano}
The following are some classical examples of Fano varieties admitting a Coble type cubic hypersurface.
\begin{enumerate}
\item The four Severi varieties each admits a unique Coble type cubic as its defected secant variety \cite[Chapter~IV]{Zak} (see also \cite{Landsberg-Manivel}), namely:
\begin{enumerate}
\item The Veronese surface $v_2(\bP^2)\subset\bP^5$ and the discriminant cubic;
\item The Segre fourfold $\bP^2\times \bP^2\subset \bP^8$ and the determinantal cubic;
\item The Grassmannian $\Gr(2,6)\subset \bP^{14}$ and the Pfaffian cubic;
\item The $E_6$-Grassmannian $E_6\subset \bP^{26}$ (also known as the \emph{Cayley plane}) and the \emph{Cartan cubic}.
\end{enumerate}
In each case, the homogeneous ideal is generated by the polar quadrics.
\item The $5$-dimensional moduli space $M_8\coloneqq (\bP^1)^8\git \SL(2)$ of $8$ ordered points on $\bP^1$ admits a unique Coble type cubic in $\bP^{13}$ \cite{Howard-Millson-Snowden-Vakil}.
Its homogeneous ideal is generated by the $14$ polar quadrics.
The singular locus of $M_8$, or equivalently, the strictly semistable locus for the $\SL(2)$-action, is given by $\frac12\binom{8}{4}=35$ isolated points.%
\footnote{
We did not check whether the singularities of $M_8$ are locally complete intersection.
This will not be relevant below when we take general $\bP^{11}$-linear sections that are smooth.
}
\end{enumerate}
\end{example}
Now we look at some K-trivial examples, which will include the classical Coble hypersurfaces and our new examples.

\begin{example}\label{example:Ktrivial}
We start with the simplest case of elliptic curves.
\begin{enumerate}
\item An elliptic normal quartic $E$ in $\bP^3$ is the complete intersection of two quadrics.
We obtain a net of (decomposable) Coble type quartics; a general member in this net is singular exactly along $E$, while along a conic the quartics degenerate to the square of a quadric.
\item An elliptic normal sextic $E$ in $\bP^5$ can be realized as a $\bP^5$-section of the Segre fourfold $\bP^2\times\bP^2\subset \bP^8$ and therefore admits a determinantal Coble type cubic.
In fact, there exists a pencil of such cubics, with $4$ of them singular along the entire Veronese surface \cite[Section~6.3]{Laza:moduli}.
\end{enumerate}
\end{example}

\begin{example}The constructions for elliptic curves can be generalized as follows.
\begin{enumerate}
\item For a smooth complete intersection of $n+1$ quadrics in $\bP^{2n+1}$, a general member in $\Sym^2 I_2$ is a Coble type quartic.
\item For the Severi varieties $\bP^2\times \bP^2$, $\Gr(2,6)$ and $E_6$, as well as the fivefold $M_8$ in \autoref{example:fano}, by the adjunction formula and a Bertini type argument (see for example \cite[Teorema~2.8]{Ottaviani:codim}), a suitable linear section provides an example of a smooth K-trivial variety admitting a Coble type cubic.%
\footnote{One deduces from the Lefschetz hyperplane theorem that they are \emph{strict Calabi--Yau}, that is, $\omega_X\simeq \cO_X$ and $h^i(\cO_X)=0$ for $0<i<\dim X$.}
\end{enumerate}
\end{example}

\begin{example}
We list the classical Coble hypersurfaces and our new examples for hyperkähler fourfolds, together with some degenerations.
\begin{enumerate}
\item For a non-hyperelliptic genus $3$ curve $C$, the Kummer variety $K(JC)\coloneqq JC/\langle\pm1\rangle$ admits a unique Coble quartic%
\footnote{
We note that $K(JC)$ indeed satisfies the condition \eqref{eq:setting}, since it suffices to check that the quotient singularities $\bC^3/{\pm1}$ are locally complete intersection.}
in $\bP^7$.
When $C$ becomes hyperelliptic, the quartic degenerates to the square of a quadric, and $K(JC)$ can be recovered as the singular locus of a quartic section on the quadric \cite[Theorem~3]{Desale-Ramanan} (see also \cite[Remark~5.18]{Benedetti-Bolognesi-Faenzi-Manivel:Coble-quadric}).
\item A undecomposable principally polarized abelian surface $A$ admits a unique Coble cubic in $\bP^8$.
When $A=E_1\times E_2$ becomes decomposable, the cubic degenerates to the determinantal cubic singular along the Segre embedded $\bP^2\times \bP^2\subset \bP^8$, and $A$ can be recovered as the singular locus of a decomposable cubic section on the Segre fourfold.
\item \autoref{thm:coble-quartic} states that a general $(X,H)\in \cM_4^{(1)}$ admits a unique Coble type quartic in $\bP^9$.
We will see that when $(X,H)$ specializes to $(S^{[2]}, L_2-2\delta)$, where $(S,L)$ is a polarized K3 surface of genus $7$, the quartic degenerates to the square of a quadric. Moreover, $S^{[2]}$ is the singular locus of a quartic section on the quadric (\cite[Theorem~C]{projective-models}; see \autoref{subsec:genus-7}).
\item \autoref{thm:coble-cubic} states that a general $(X,H)\in \cM_6^{(1)}$ admits a unique Coble type cubic in $\bP^{14}$.
Specializing $(X,H)$ to $(S^{[2]}, L_2-2\delta)$ for $(S,L)$ a polarized K3 of genus $8$, the cubic degenerates to the Pfaffian cubic, and $S^{[2]}$ is the singular locus of a cubic section on the Grassmannian $\Gr(2,6)$ (\cite[Theorem~D]{projective-models}; see \autoref{subsec:genus-8}).
\end{enumerate}
\end{example}

The following table summarizes the given examples of K-trivial varieties of dimension $\leq 4$ which admit Coble type quartics and cubics:

\begin{table}[h!]
\label{table:examples}
\renewcommand\arraystretch{1.5}
\begin{center}
\begin{NiceTabular}{ccccc}[hvlines]
\ & \Block{1-2}{quartics} &\ & \Block{1-2}{cubics}  \\  
dim                    & ambient              & example                           & ambient                   & example                              \\
1                  & $\bP^3$              & $(2^2)$-c.i.                      & $\bP^5$                 & $(\bP^2\times\bP^2)\cap \bP^5$          \\
\Block{2-1}{2}   & \Block{2-1}{$\bP^5$} & \Block{2-1}{$(2^3)$-c.i.=K3}      & \Block{2-1}{$\bP^8$}    & $\Gr(2,6)\cap \bP^8$=K3                 \\
                       &                      &                                   &                         & $A\in\cA_2$                             \\
\Block{2-1}{3} & \Block{2-1}{$\bP^7$} & $(2^4)$-c.i.=$\CY_3$              & \Block{2-1}{$\bP^{11}$} & \Block{2-1}{$M_8\cap \bP^{11}$=$\CY_3$} \\
                       &                      & $K(A)$ for $A\in \cA_3$           &                         &                                         \\
\Block{2-1}{4}  & \Block{2-1}{$\bP^9$} & $(2^5)$-c.i.=$\CY_4$              & \Block{2-1}{$\bP^{14}$} & $E_6\cap \bP^{14}$=$\CY_4$              \\
                       &                      & $X\in \cM_{4}^{(1)}$=$\KKK^{[2]}$ &                         & $X\in \cM_{6}^{(1)}$=$\KKK^{[2]}$       \\
\end{NiceTabular}
\end{center}
\medskip
\caption{Examples of K-trivial varieties admitting Coble type quartics and cubics}
\end{table}

\begin{remark}
In each dimension $n$, the condition of \autoref{cor:dim-codim-condition} for K-trivial varieties admitting a Coble type hypersurface of degree $d$ in $\bP^N$ includes the following four possibilities for $(N,d)$: $(n+1,2n+4)$, $(n+2,n+3)$, $(2n+1,4)$, and $(3n+2,3)$.
The first three cases can always be realized by, respectively, a non-reduced double hypersurface, a reducible hypersurface for an $(m,n+3-m)$-complete intersection, and a quartic for a $(2^{n+1})$-complete intersection.
The cubic case seems more special and we are not aware of an infinite family of examples like in the other cases.
\end{remark}

We close this subsection by showing that \autoref{thm:coble-quartic} and \autoref{thm:coble-cubic} cannot be generalized to other moduli spaces of polarized hyperkähler varieties for the known deformation types: 

\begin{proposition}
\label{prop:HK-coble}
Let $(X,H)$ be a general polarized hyperkähler manifold of deformation type $\KKK^{[m]}$, $\Kum_m$, $\OG6$, or $\OG10$.
If $X$ is embedded via the complete linear system $|H|$,
then $X$ can admit a Coble type hypersurface if and only if we are in the following cases:
\begin{itemize}
\item $(X,H)$ is a polarized K3 surface of genus $3$, $4$, $5$, or $8$.
\item $(X,H)$ is a polarized fourfold of $\KKK^{[2]}$-type in the two families $\cM_4^{(1)}$ or $\cM_6^{(1)}$.
\end{itemize}
\end{proposition}
\begin{proof}
The Riemann--Roch polynomial is known for the known deformations types \cite{Rios-Ortiz}, which can be used to compute the number of global sections $h^0(X,H)$ by the Kodaira vanishing theorem.
Hence if $(X,H)$ is of dimension $n\coloneqq 2m$ with Beauville--Bogomolov--Fujiki square $q_X(H)=2k$, and it admits a Coble type hypersurface of degree $d$ when embedded via the complete linear system $|H|$, then by \autoref{cor:dim-codim-condition} we would have either
\[
\binom{k+m+1}{m}-1 = 2m + \frac{2(2m+1)}{d-2} \text{ for $\KKK^{[m]}$ and $\OG10$}
\]
or
\[
(m+1)\binom{k+m}{m}-1 = 2m + \frac{2(2m+1)}{d-2} \text{ with $m\ge2$ for $\Kum_m$ and $\OG6$}.
\]
In either case, the right hand side is bounded above by $6m+2$, while the left hand side is a binomial coefficient that grows in $m$ as a polynomial of degree $k+1$.
In particular, one can check that when $k\gg 0$ (say when $k\ge 8$), or when $d\gg 0$ for a fixed $k$, neither equality can hold for $m\ge 1$.
Hence it suffices to enumerate a finite number of cases, and one concludes that the second case never occurs, while the first case can happen for the following values
\[
\begin{aligned}
m=1&,(k,d)\in\set{(7,3),(4,4),(3,5),(2,8)},\\
m=2&,(k,d)\in\set{(3,3),(2,4),(1,12)},\\
m=4&,(k,d)\in\set{(1,5)}.
\end{aligned}
\]

When $m=1$, so $X$ is a K3 surface, it is not difficult to check that $X$ indeed admits a Coble type hypersurface in all four cases.

When $m=2$, so $X$ is of $\KKK^{[2]}$-type, taking into consideration the divisibility the four possible cases are the general members of the moduli spaces $\cM_6^{(1)}$, $\cM_6^{(2)}$, $\cM_4^{(1)}$, and $\cM_2^{(1)}$.
The general element in $\cM_2^{(1)}$ is a double EPW sextic, so it can be excluded since it is not very ample. On the other hand, $\cM_6^{(2)}$ is the family of varieties of lines of cubic fourfolds, which are projectively normal\footnote{This can be checked using the Koszul complex.} in $\bP^{14}$; they are precisely contained in $15$ quadrics, namely the quadrics containing the Grassmannian $\Gr(2,6)$, and so any cubic hypersurface singular along $X$ must also be singular along $\Gr(2,6)$. This discards the case  $\cM_6^{(2)}$.

When $m=4$, so $X$ is of $\KKK^{[4]}$-type, we are left with the two cases $\cM_{\KKK^{[4]},2}^{(1)}$ and $\cM_{\KKK^{[4]},2}^{(2)}$.
The latter is the family of LLSvS eightfolds, whose primitive polarization is not very ample.
For the former, by \cite[Corollary~4.7]{Debarre:survey}, the general member $(X,H)$ admits an antisymplectic involution $\sigma$, and the linear system $|H|$ factors through $X/\sigma$ so it is not very ample.%
\footnote{One could still ask whether the image of $X$ admits a Coble type quintic; for comparison, for the double EPW sextics, since the image is a sextic hypersurface, it does admit a non-reduced Coble type hypersurface of degree $12$.
Note also that the EPW sextic itself is a Coble type sextic for a surface of general type.
}
\end{proof}

\subsection{Behavior under deformation}
\label{subsec:coble-deformation}

It is worth mentioning that, in all the K-trivial examples, the homogeneous ideal of $X$ is \emph{not} generated by the polars of the Coble type hypersurface: for complete intersections of $n+1$ quadrics, the homogenous ideal is clearly generated by the quadrics; for the examples of $\bP^N$-linear sections of a Fano $G\subset\bP^M$, the variety $X\subset \bP^N$ is cut out scheme-theoretically by the $(N+1)$ polar quadrics of the Coble type cubic, however its homogeneous ideal is generated by the restriction of all $(M+1)$ polar quadrics before taking the linear section.

For the examples of Coble and our hyperkähler examples, again the homogeneous ideal is not generated by the polar hypersurfaces.
In fact, one might have the following curious observation: the number of extra equations in degree $d$ that are required to generate the saturated homogeneous ideal is equal to the number of moduli of $X$.
Namely,
\begin{itemize}
\item For $(A,3\Theta)$, one needs $3=\dim \cA_2$ more cubics \cite{Barth};
\item For $(K(A),2\Theta)$, one needs $6=\dim \cA_3$ more quartics \cite[Remark~2.1]{Ren-Sam-Schrader-Sturmfels};
\item For $(X,H)\in \cM_4^{(1)}$, one needs $20=\dim \cM_4^{(1)}$ more quartics (established in \cite[Theorem~E.(1)]{projective-models} but also follows from the same argument as \autoref{cor:square-6-homogeneous-ideal});
\item For $(X,H)\in \cM_6^{(1)}$, one needs $20=\dim \cM_6^{(1)}$ more cubics as well (see \autoref{cor:square-6-homogeneous-ideal}).
\end{itemize}

To understand this remarkable fact, we need to examine how the degree-$d$ equations of $X$ are related to first-order deformations of $X$.

Let $X\subset \bP^N=\bP^{n+c}$ be a variety of dimension $n$ satisfying the condition~\eqref{eq:setting} and admitting a Coble type hypersurface $Y=Z(f)$ of degree $d$ in $\bP^N=\bP(V_{N+1})$.
We will consider the following diagram
\begin{equation}
\label{eq:coble-deformation-diagram}
\begin{tikzcd}
0\ar[r]& \cO_{\bP^N}\ar[r]\ar[d,"\cdot f"] & V_{N+1}\otimes \cO_{\bP^N}(1) \ar[r,twoheadrightarrow]\ar[d,twoheadrightarrow,"\partial f"] & \cT_{\bP^N}\ar[d,twoheadrightarrow,"\psi"]\ar[r] & 0\\
0\ar[r]& \cI^2_X(d)\ar[r] & \cI_X(d) \ar[r,twoheadrightarrow] & \cN_{X/\bP^N}^\vee(d)\ar[r] & 0
\end{tikzcd}
\end{equation}
where the left square commutes by Euler's theorem for homogeneous polynomials, and induces a surjective map $\psi\colon\cT_{\bP^N} \to \cN^\vee_{X/\bP^N}(d)$.

\begin{lemma}
The map $\psi$ in the diagram \eqref{eq:coble-deformation-diagram} composed with the conormal map $\cN^\vee_{X/\bP^N}(d)\to\Omega_{\bP^N}|_X(d)$ defines a quadratic form on $\cT_{\bP^N}|_X$ that induces the one in \eqref{eq:coble-quadratic-form}.
Consequently, the following diagram is commutative
\[
\begin{tikzcd}
\cT_{\bP^N} \ar[r,equal]\ar[d,twoheadrightarrow,"\psi"] & \cT_{\bP^N}\ar[d]\\
\cN^\vee_{X/\bP^N}(d) \ar[r,"\sim"] & \cN_{X/\bP^N}
\end{tikzcd}
\]
where the right arrow is the normal map, and the bottom isomorphism is induced by the quadratic form.
\end{lemma}
\begin{proof}
We pick a basis $(e_i)_{0\le i\le N+1}$ for $V_{N+1}$ and denote by $(dx_i)_{0\le i\le N+1}$ the corresponding one-forms.
Using the Euler sequence, we can write a local section of $\cT_{\bP^N}$ as $\sum_i e_i\otimes s_i$ for some regular functions $s_i$ (seen as local sections of $\cO_{\bP^N}(1)$).
The map $\psi$ is given by
\[
\sum_i e_i\otimes s_i\mapsto  \sum_i\tfrac{\partial f}{\partial x_i}\cdot s_i.
\]
On the other hand, the conormal map
\[
\cN^\vee_{X/\bP^N}(d)\cong \cI_{X}/\cI^2_{X}(d)\cong\cI_{X}|_X(d)\to\Omega_{\bP^N}|_X(d)
\]
is by definition the differential map restricted to $X$
\[
g \mapsto \sum_j \left.\tfrac{\partial g}{\partial x_j}\right|_X\cdot dx_j.
\]
Hence when composing the two maps, we get
\[
\sum_i e_i\otimes s_i\mapsto  \sum_{i,j}\left.\left(\tfrac{\partial^2 f}{\partial x_i\partial x_j}\cdot s_i+\tfrac{\partial f}{\partial x_i}\cdot\tfrac{s_i}{\partial x_j}\right)\right|_X\cdot dx_j = \sum_{i,j}\left.\tfrac{\partial^2 f}{\partial x_i\partial x_j}\right|_X\cdot s_i\cdot dx_j,
\]
where we used the fact that the polars $\frac{\partial f}{\partial x_i}$ all vanish along $X$.
This is clearly a symmetric quadratic form on $\cT_{\bP^N}|_X$ that induces the one on $\cN_{X/\bP^N}$ given in \eqref{eq:coble-quadratic-form}.
\end{proof}

Note that the commutativity of the above diagram shows that taking global sections of the map $\psi$ is compatible with the natural map
\begin{equation}
\label{eq:tangent-to-first-order}    
H^0(\bP^N,\cT_{\bP^N}) \to H^0(X,\cN_{X/\bP^N}).
\end{equation}
In view of the diagram \eqref{eq:coble-deformation-diagram}, we obtain a natural map
\begin{equation}
\label{eq:degree-d-to-first-order}
H^0(\bP^N,\cI_X(d)) \to H^0(X,\cN_{X/\bP^N})
\end{equation}
comparing the space $H^0(\bP^N, \cI_X(d))$ of degree-$d$ equations of $X$ and the space $H^0(X,\cN_{X/\bP^N})$ of first-order deformations of $X$ in $\bP^N$.

We have the following immediate consequences.
\begin{corollary}
\label{cor:no-syzygy}
In the same setting as above,
\begin{enumerate}
\item \label{cor:no-syzygy-kernel-KS}
Consider the subspace $\langle\tfrac{\partial f}{\partial x_i}\cdot x_j\rangle$ of degree-$d$ equations generated by the polars $\frac{\partial f}{\partial x_i}$ of $f$.
Under the map \eqref{eq:degree-d-to-first-order}, its image is contained in the kernel of the Kodaira--Spencer map
\[
\delta_{X/\bP^N}\colon H^0(X, \cN_{X/\bP^N})\to H^1(X, \cT_X).
\]
\item\label{cor:no-syzygy-item} If we assume that the map \eqref{eq:tangent-to-first-order} is injective,
then there is no linear syzygy among the polars $\frac{\partial f}{\partial x_i}$.
More generally, if we assume that $H^0(\cI^2_X(d))=\bC f$, then we can identify the space of linear syzygies with the kernel of~\eqref{eq:tangent-to-first-order}
\item\label{cor:no-syzygy-iso}
If we assume that $H^0(\cI^2_X(d))=\bC f$ and $H^1(\cI^2_X(d))=0$, then we can identify the quotient
$
H^0(\cI_X(d))/\langle\tfrac{\partial f}{\partial x_i}\cdot x_j\rangle
$
with the cokernel of \eqref{eq:tangent-to-first-order}.
\end{enumerate}
\end{corollary}
\begin{proof}
The first point follows easily from a diagram chase.

For the second point, we consider the global sections of the diagram \eqref{eq:coble-deformation-diagram}
\[
\begin{tikzcd}
0\ar[r]& H^0(\cO_{\bP^N})\ar[r]\ar[d,"\cdot f",hookrightarrow] & V_{N+1}\otimes H^0(\cO_{\bP^N}(1)) \ar[r,twoheadrightarrow]\ar[d,"x\otimes \ell\mapsto \partial_x f\cdot \ell"] & H^0(\cT_{\bP^N})\ar[d]\ar[r] & 0\\
0\ar[r]& H^0(\cI^2_X(d))\ar[r] & H^0(\cI_X(d)) \ar[r] & H^0(X,\cN_{X/\bP^N})
\end{tikzcd}
\]
The two rows are exact while the left column is injective.
Hence if the right column is injective by assumption, the middle column would be injective as well, which means precisely that there is no linear syzygy among the polar hypersurfaces.

The second part of (\ref{cor:no-syzygy-item}), as well as (\ref{cor:no-syzygy-iso}) follow similarly from the snake lemma.
\end{proof}
\begin{remark}
Another way to see \autoref{cor:no-syzygy}.(\ref{cor:no-syzygy-kernel-KS}) is that, under the map \eqref{eq:degree-d-to-first-order}, the degree-$d$ equations generated by the polars of $f$ correspond to first-order deformations of $X$ induced by first-order automorphisms of the ambient $\bP^N$: clearly such deformations will lie in the kernel of $\delta_{X/\bP}$.
\end{remark}
The map \eqref{eq:degree-d-to-first-order} also has a geometric interpretation in view of the following lemma.
\begin{lemma}
\label{lem:tangential}
Consider a one-parameter family $\cY\to \Delta$ of hypersurfaces in $\bP^N$ given by the equation $\cY_t=Z(F_t)$ for $t\in \Delta$.
Denote by $\cZ^*\coloneqq \Sing(\cY^*/\Delta^*)$ the relative singular locus away from the origin $0\in\Delta$.
Let $\cZ_0$ be the closure $\overline{\cZ^*}\cap \cY_0$ and $f_1\coloneqq \frac{\partial F_t}{\partial t}$ be the tangential hypersurface.
Then the hypersurface $Z(f_1)\subset \bP^N$ contains $\cZ_0$ set-theoretically.
\end{lemma}
\begin{proof}
Assume that $x\in \cZ_0$ is not contained in the tangential hypersurface $Z(f_1)$, in other words, $f_1(x)=\frac{\partial F_t}{\partial t}(x)\ne 0$.
Since $\cY\subset \bP^N\times \Delta$ is a hypersurface, by the Jacobian criterion, we readily see that the total space $\cY$ is smooth at the point $x$.
Let $\cU\subset \cY$ be a smooth open subset containing $x$.
Then by generic smoothness, up to shrinking the base $\Delta$, we may assume that $\cU\cap \cY^*\to\Delta^*$ is a smooth morphism.
Since this contradicts the assumption that $x$ lies in the closure $\cZ_0$ of the relative singular locus $\cZ^*$, the proof is complete.
\end{proof}
Namely, for a first-order deformation $\cX$ of $X=\cX_0$, if one can construct a non-trivial family of degree-$d$ hypersurfaces $\cY\to\Delta$ such that the relative singular locus $\cZ\to \Delta$ contains $\cX$, then the tangential hypersurface provides a degree-$d$ equation of $X$.
In particular, such a first-order deformation would be contained in the image of the map \eqref{eq:degree-d-to-first-order}.
Therefore, assuming that all deformations of $X$ admit a Coble type hypersurface, the number of extra equations in degree $d$ would be at least the number of moduli of $X$.

\begin{remark}
Notably, in the classical case of $A\subset \bP^8$, the $4$-dimensional space of cubics (the Coble cubic and three suitable extra cubics, namely, those invariant by the \emph{Heisenberg group} action; see \cite{Beauville}) can be geometrically interpreted as the embedded tangent space to the $3$-dimensional \emph{Burkhardt quartic} $\cB\subset\bP^4$, a parameter space for Coble cubics.
Similarly, for $K(A)\subset \bP^7$, the $7$-dimensional space of quartics (the Coble quartic and six extra Heisenberg-invariant ones) can be interpreted as the embedded tangent space to the $6$-dimensional \emph{Göpel variety} \cite{Ren-Sam-Schrader-Sturmfels,Benedetti-Bolognesi-Faenzi-Manivel:Gopel}.
\end{remark}

\section{Proof of \autoref{thm:coble-quartic}}
\label{sec:thmA}
In this section we prove the existence of the Coble type quartic hypersurface in the case of square $4$.
The key point of our argument is a specialization to the Hilbert square of a very general polarized K3 surface of genus $7$, for which we first give a brief account.

\subsection{Hilbert squares of K3 surfaces of genus \texorpdfstring{$7$}{7}}
\label{subsec:genus-7}
Let $(S,L)$ be a polarized K3 surface of genus $7$ with $\Pic(S)=\bZ\cdot L$. The line bundle $L_2-2\delta$ on $S^{[2]}$ has Beauville--Bogomolov square $4$ and divisibility $1$ and, as shown in \cite[Theorem A]{projective-models}, it is very ample. We review some aspects of the corresponding embedding
\begin{equation}\label{eq:S2-genus7}
    S^{[2]}\into \bP^9=\bP H^0(L_2-2\delta)^\vee,
\end{equation}
whose geometry is described in \cite[Section 4]{projective-models}, where the reader is referred for further details.

We fix a $10$-dimensional vector space $V_{10}$ equipped with a non-degenerate quadratic form~$q$.
Let $Q\coloneqq Z(q)\subset \bP(V_{10})$ be the smooth $8$-dimensional quadric hypersurface, and let $\OGr_+\coloneqq\OGr(5,V_{10})\subset \bP^{15}=\bP(V_+)$ be one of the two orthogonal Grassmannians (in its spinor embedding) parametrizing maximal isotropic subspaces of $q$.

Mukai \cite{mukai-models} proved that the embedding 
\[
S\into \bP(H^0(L)^\vee)=\bP^7
\]
is a $\bP^7$-linear section of $\OGr_+$. In particular, $(S,L)$ is determined by the $8$-dimensional vector subspace
$H^0(L)^\vee\into V_+$.
We denote by $W$ the $8$-dimensional quotient
\[
W\coloneqq V_+/H^0(L)^\vee.
\]
It turns out that there is a canonical identification $\bP(V_{10})\cong \bP H^0(L_2-2\delta)^\vee$, under which \eqref{eq:S2-genus7} identifies with a closed immersion
\[
S^{[2]}\into \bP(V_{10})
\]
factoring through the quadric $Q$. Furthermore, we can consider the induced morphism of rank-$8$ vector bundles
\[
\varphi\colon \cS_+\into V_+\otimes \cO_Q\onto W\otimes \cO_Q=\cO_Q^{\oplus 8}
\]
on $Q$, where $\cS_+$ is one of the spinor bundles. Notably, the embedded $S^{[2]}\subset \bP^9$ coincides (scheme-theoretically) with the rank-$6$ degeneracy locus $D_6(\varphi)$.

This incarnation of $S^{[2]}$ as a degeneracy locus has several consequences. For instance, with the help of the \emph{Gulliksen--Negård complex}, a four-term locally free resolution for $\cI_{S^{[2]}/Q}$ in $Q$, one can obtain a precise description of the equations of $S^{[2]}$ in $\bP^9$.

\begin{proposition}[{\cite[Section~4.5]{projective-models}}]\label{prop:eqgenus7}
The following hold:
\begin{enumerate}
    \item $Q$ is the unique quadric hypersurface in $\bP^9$ containing $S^{[2]}$. In other words, the natural map between two $55$-dimensional vector spaces
    \[
    m_2\colon H^0(\cO_{\bP^9}(2))\to H^0(\cO_{S^{[2]}}(2))
    \]
    has corank $1$.

    \item\label{prop:eqgenus7-2} For every $d\geq3$, the natural restriction map
    \[
    m_d\colon H^0(\cO_{\bP^9}(d))\to H^0(\cO_{S^{[2]}}(d))
    \]
    is surjective, namely $H^1(\cI_{S^{[2]}/\bP^9}(d))=0$. In particular, $H^0(\cI_{S^{[2]}/\bP^9}(3))$ is given by the $10$-dimensional vector space of homogeneous cubic polynomials divisible by $q$.
\end{enumerate}
\end{proposition}

Moreover we have the following result on the first-order deformations of $S^{[2]}$ in $Q$.

\begin{proposition}[{\cite[Proposition~4.19]{projective-models}}]\label{prop:64genus7}
There is a canonical identification
\[
H^0(\cN_{S^{[2]}/Q})\cong H^0(L)\otimes W.
\]
\end{proposition}

In this way, $H^0(\cN_{S^{[2]}/Q})$ is canonically identified with the Zariski tangent space to the Grassmannian $\Gr(8, V_+)$ at the point $[H^0(L)^\vee]$. In other words, a first-order deformation of $S^{[2]}$ in $Q$ is the same as a first-order deformation of the chosen $\bP^7$ (determining the K3 surface $S$) in the ambient $\bP^{15}$ for the orthogonal Grassmannian $\OGr_+$. This is the key observation to establish the following:

\begin{proposition}
[{\cite[Proposition 4.22]{projective-models}}]\label{prop:deformation-genus7}
A general member $(X,H)\in \cM_4^{(1)}$ in $\bP^9$ is not contained in any quadric hypersurface of $\bP^9$.
\end{proposition}

In the sequel, we assume that $(S,L)$ is sufficiently general among polarized K3 surfaces of genus $7$ with Picard rank $1$, so that $S^{[2]}=D_6(\varphi)$ coincides with the singular locus of the degeneracy locus $D_7(\varphi)\subset Q$.
This follows from a Bertini-type argument (see \cite[Teorema~2.8]{Ottaviani:codim}), as $\cS_+^\vee$ is globally generated and $\Hom(\cS_+,W\otimes \cO_Q)=\Hom(H^0(L),W)$. Note that
\[
D_7(\varphi)=Z(\det \varphi)\subset Q
\]
is a divisor in the linear system $|\cO_Q(4)|$. 

It is remarkable that the two normal bundles $\cN_{S^{[2]}/Q}$ and $\cN_{Q/\bP}$ both satisfy the relation $\cN\cong \cN^\vee(4)$.
This can be easily checked by hand, but more conceptually follows from the existence of a nowhere degenerate quadratic form by \autoref{prop:coble-quadratic-form}.

\subsection{Proof of \autoref{thm:coble-quartic}}
We begin with the following preparatory lemma, which shows the uniqueness of the quartics $D_7(\varphi)\subset Q$ and $Z(q^2)\subset \bP^9$ singular along $S^{[2]}$.

\begin{lemma}\label{prep-lemma}
    Consider $(S^{[2]},L_2-2\delta)$ for $(S,L)$ a general polarized K3 surface of genus $7$, and let $Q\coloneqq Z(q)\subset \bP^9\coloneqq\bP(H^0(S^{[2]},L_2-2\delta)^\vee)$ be the unique quadric containing $S^{[2]}$. Then:
    \begin{enumerate}
        \item\label{prep-lemma-1} $H^0(Q,\cI_{S^{[2]}/Q}^2(4))\cong\bC$, and is generated by the first degeneracy locus $D_7(\varphi)\in|\cO_Q(4)|$.
        Furthermore, we have the vanishing $H^1(Q,\cI_{S^{[2]}/Q}^2(4))=0$. 
        \item\label{prep-lemma-2} $H^0(\bP^9,\cI_{S^{[2]}/\bP}^2(4))\cong\bC$, and is generated by the non-reduced quartic $Z(q^2)\in |\cO_{\bP}(4)|$.
        Furthermore, we have the vanishing $H^1(\bP^9,\cI_{S^{[2]}/\bP}^2(4))=0$.
    \end{enumerate}
\end{lemma}
\begin{proof}
Part \eqref{prep-lemma-1} follows from \cite[Remark 4.15]{projective-models} by observing that, under the isomorphism $\cN_{S^{[2]}/Q}^\vee(4)\cong \cN_{S^{[2]}/Q}$ followed by the identification $H^0(\cN_{S^{[2]}/Q}) \cong H^0(L)\otimes W$ (given in \autoref{prop:64genus7}), the surjection
\[
H^0(Q,\cI_{S^{[2]}/Q}(4))\onto H^0(L)\otimes W
\]
corresponds to the global sections of the natural map $\cI_{S^{[2]}/Q}(4)\onto \cN_{S^{[2]}/Q}^\vee(4)$.

Now we prove \eqref{prep-lemma-2}.
For a quartic hypersurface $F\coloneqq Z(f)\in |\cI^2_{S^{[2]}/\bP}(4)|$ that is singular along $S^{[2]}$, its partial derivatives $\frac{\partial f}{\partial x_i}$ must all lie in $H^0(\cI_{S^{[2]}/\bP}(3))$, which coincides with $H^0(\cI_{Q/\bP}(3))=H^0(\cO_{\bP}(1))\cdot q$ in view of \autoref{prop:eqgenus7}.(\ref{prop:eqgenus7-2}).
Therefore, $F$ is singular along $Q$ as well and so it must be $Z(q^2)$.
For the vanishing $H^1(\bP^9,\cI_{S^{[2]}/\bP}^2(4))=0$, we examine the following commutative diagram with exact rows and columns
\[
\begin{tikzcd}
 &  0\ar[d] & 0\ar[d] & 0\ar[d] & \\
 0\ar[r] & \cI_{S^{[2]}/\bP}(2)\ar[r]\ar[d,"\cdot q"] & \cO_\bP(2)\ar[r,"m_2"]\ar[d] & \cO_{S^{[2]}}(2) \ar[r]\ar[d] & 0\\
0\ar[r] & \cI_{S^{[2]}/\bP}^2(4)\ar[r]\ar[d] & \cI_{S^{[2]}/\bP}(4)\ar[r]\ar[d] & \cN_{S^{[2]}/\bP}^\vee(4) \ar[r]\ar[d] & 0\\
0\ar[r] & \cI_{S^{[2]}/Q}^2(4)\ar[r]\ar[d]& \cI_{S^{[2]}/Q}(4)\ar[d]\ar[r] & \cN_{S^{[2]}/Q}^\vee(4)\ar[d]\ar[r] & 0\\
 & 0 & 0 & 0 & 
\end{tikzcd}
\]
By looking at the first column, we obtain two isomorphisms
$H^0(\cI_{S^{[2]}/\bP}(2))\simto H^0(\cI^2_{S^{[2]}/\bP}(4))$ that maps $q$ to $q^2$,
and $H^0(\cI^2_{S^{[2]}/Q}(4))\simto H^1(\cI_{S^{[2]}/\bP}(2))$ that maps the hypersurface $D_7(\varphi)$ to a generator of the cokernel of the restriction map $m_2\colon H^0(\cO_\bP(2))\to H^0(\cO_{S^{[2]}}(2))$.
Since $H^2(\cI_{S^{[2]}/\bP}(2))=0$, we can conclude that $H^1(\bP^9,\cI^2_{S^{[2]}/\bP}(4))=H^1(Q,\cI^2_{S^{[2]}/Q}(4))=0$.
\end{proof}

\begin{proof}[Proof of \autoref{thm:coble-quartic}]
Now we consider a general $(X,H)\in\cM_4^{(1)}$. We want to prove the existence of a unique quartic hypersurface $Y\subset \bP^9=\bP(H^0(X,H)^\vee)$ such that $X=\Sing(Y)$ scheme-theoretically.

First we prove that $h^0(\bP^9,\cI_{X/\bP}^2(4))=1$, namely that there is a unique quartic $Y\subset\bP^9$ with $X\subset\Sing(Y)$. To this end, we consider the short exact sequence
\[
0\to \cI_{X/\bP}^2(4)\to \cI_{X/\bP}(4)\to \cN_{X/\bP}^\vee(4)\to 0.
\]
Quartic normality of $X$ in $\bP^9$ tells us that $h^0(\bP^9,\cI_{X/\bP}(4))=120$, $h^1(\bP^9,\cI_{X/\bP}(4))=0$.
Thus we have
\begin{equation}
\label{eq:120}
h^0(\bP^9,\cI_{X/\bP}^2(4)) + h^0(X,\cN_{X/\bP}^\vee(4)) = 120 + h^1(\bP^9,\cI_{X/\bP}^2(4)).
\end{equation}
Now we let $(X,H)$ specialize to $(S^{[2]},L_2-2\delta)$, for $(S,L)$ a very general polarized K3 surface of genus~$7$.
In this case, \eqref{eq:120} still holds since $S^{[2]}$ is quartically normal in $\bP^9$ by \autoref{prop:eqgenus7}, and we have $h^0(\bP^9,\cI_{S^{[2]}/\bP}^2(4))=1$ and $h^1(\bP^9,\cI_{S^{[2]}/\bP}^2(4))=0$ by \autoref{prep-lemma}.(\ref{prep-lemma-2}).
Using upper semicontinuity of cohomology ranks and in view of the relation \eqref{eq:120}, these equalities must hold for a general $(X,H)\in\cM_4^{(1)}$ as well.

Next we show that $Y$ is singular \emph{exactly} along $X$.
Consider a general one-parameter flat family
\[
\begin{tikzcd}
\cX\ar[d]\subset \bP^9\times\Delta\\
\Delta
\end{tikzcd}
\]
with central fiber $\cX_0=S^{[2]}$ (embedded in $\bP^9$ via $L_2-2\delta$), where $(S,L)$ is a general polarized K3 surface of genus~$7$.
Assume that $\cX_0$ is contained in the standard quadric $Q=Z(q)$ with $q\coloneqq x_0^2+\dots +x_9^2$.
Recall from \autoref{prep-lemma}.(\ref{prep-lemma-2}) that $q^2$ is the unique quartic in $\bP^9$ which is singular along $\cX_0$.

Without loss of generality, we assume that the base $\Delta$ is the formal scheme $\Spec \bC[[t]]$ and denote by $\eta$ its generic point.
We obtain a quartic hypersurface $\cY\coloneqq Z(F)\subset \bP^9\times\Delta$, where
\[
F\coloneqq q^2+t\cdot f_1+t^2\cdot f_2 +\dots \in H^0(\bP^9,\cO(4))\otimes_\bC \bC[[t]],
\]
such that $\cY_\eta$ is the unique quartic hypersurface in $\bP^9_{k(\eta)}$ whose singular locus contains $\cX_\eta$.
It suffices to show that $\Sing(\cY_\eta)$ coincides with $\cX_\eta$.
Note that, by the generality assumption, the quartic $f_1$ is not a multiple of $q$.
Otherwise, one may write $f_1=4q\cdot \sum_{i=0}^9\ell_i\cdot x_i$, where $\ell_i=\ell_i(x)$ are linear forms, and consider the linear map $\ell\colon x\mapsto (\ell_0(x),\dots,\ell_9(x))$.
Then we may check that in first-order we have
\[
q^2(x) + \varepsilon \cdot f_1(x) = q^2(x+\varepsilon \cdot \ell(x)).
\]
Therefore, up to taking the first-order automorphism $x\mapsto x+\varepsilon\cdot \ell(x)$ of $\bP^9$,
the family of quartic hypersurface we get is identically $q^2$,
whence the corresponding first-order deformation must be entirely contained in $Q$, contradicting \autoref{prop:deformation-genus7}.

Consider the intermediate subvariety
\[
\widetilde{\cZ}\coloneqq Z(\tfrac{\partial F}{\partial x_i})_{i=0,\dots,9}
\]
whose fiber $\widetilde{\cZ}_\eta$ over the generic point $\eta$ coincides with $\Sing(\cY_\eta)$, while the fiber $\widetilde{\cZ}_0$ over $0$ coincides with $\Sing(\cY_0)=Q$.
In particular, $\widetilde{\cZ}$ contains $\cX$.
Let $\cZ$ denote the schematic closure of $\widetilde{\cZ}_\eta$ in $\widetilde{\cZ}$ (namely, the union of the irreducible components of $\widetilde\cZ$ dominating $\Delta$); in particular, $\cZ$ is flat over $\Delta$.
Note that $\cZ$ contains $\cX$, since $\cX$ is the schematic closure of $\cX_\eta\subset \widetilde{\cZ}_\eta$. Therefore, $\cZ_0$ contains $\cX_0$ while being contained in $\widetilde{\cZ}_0 = Q$ .

Now observe that, for every $i$ and $j$, the polynomial
\[
\begin{aligned}
x_i\cdot \tfrac{\partial F}{\partial x_j}-x_j\cdot \tfrac{\partial F}{\partial x_i}
&=x_i \cdot \left(4 q\cdot x_j + t\cdot \tfrac{\partial f_1}{\partial x_j} + O(t^2)\right)
-x_j \cdot \left(4 q\cdot x_i + t\cdot \tfrac{\partial f_1}{\partial x_i} + O(t^2)\right)\\
&=t\cdot \left(x_i\cdot \tfrac{\partial f_1}{\partial x_j} -x_j \cdot \tfrac{\partial f_1}{\partial x_i} + O(t)\right)
\end{aligned}
\]
vanishes along $\cZ$, hence so does the polynomial
\[
x_i\cdot \tfrac{\partial f_1}{\partial x_j} -x_j \cdot \tfrac{\partial f_1}{\partial x_i} + O(t).
\]
In particular, the quartic $f_{ij}\coloneqq x_i\cdot \tfrac{\partial f_1}{\partial x_j} -x_j \cdot \tfrac{\partial f_1}{\partial x_i}$ vanishes along $\cZ_0$ for every $i$ and $j$.

Furthermore, the quartic $f_1$ also vanishes along $\cZ_0$, as Euler's identity for homogeneous polynomials implies that for every $i$
\[
\begin{aligned}
x_i\cdot 4 f_1 &= x_i \cdot \sum_{j}x_j\tfrac{\partial f_1}{\partial x_j} = \sum_{j}x_j\cdot x_i\tfrac{\partial f_1}{\partial x_j}\\
&= \sum_{j}x_j\cdot \left(x_j\tfrac{\partial f_1}{\partial x_i}+f_{ij}\right)= q\cdot\tfrac{\partial f_1}{\partial x_i}+\sum_{j}x_j\cdot f_{ij}
\end{aligned}
\]
vanishes along $\cZ_0$.

Since $\cZ_0$ contains $\cX_0=S^{[2]}$, we see that the quartic $f_1$ as well as the $45$ quartics $f_{ij}$ ($i<j$) all vanish along $S^{[2]}$. This means that the quartic hypersurface $Z(q,f_1)\in|\cO_Q(4)|$ in $Q$ is singular along $S^{[2]}$, as the quartics $2f_{ij}$ are the $2\times2$-minors of the Jacobian matrix
\begin{equation}
\label{eq:jacobian-matrix}
\begin{pmatrix}
\frac{\partial f_1}{\partial x_0}  & ... & \frac{\partial f_1}{\partial x_9} \\[1ex]
\frac{\partial q}{\partial x_0}  & ... & \frac{\partial q}{\partial x_9} 
\end{pmatrix}.
\end{equation}
Therefore, it follows from \autoref{prep-lemma}.(\ref{prep-lemma-1}) that $Z(q,f_1)$ is the first degeneracy locus $D_7(\varphi)\subset Q$. In particular, $S^{[2]}$ is exactly the singular locus of $Z(q,f_1)$, and so $S^{[2]}$ is cut out scheme-theoretically by $Q$, $Z(f_1)$, and the $45$ quartics $f_{ij}$. This implies that $\cZ_0$ coincides with $\cX_0=S^{[2]}$.
Comparing the Hilbert polynomials (using the lexicographical order), we have
\[
h_{\cX_\eta}\le h_{\cZ_\eta}= h_{\cZ_0}= h_{\cX_0},
\]
where the first inequality is due to the inclusion $\cX_\eta\subset\cZ_\eta$, and the central equality follows from the flatness of $\cZ\to \Delta$.
But since $h_{\cX_\eta}=h_{\cX_0}$ as we assumed that the family $\cX\to \Delta$ is flat, the first inequality must be an equality. This proves that $\cX_\eta=\cZ_\eta=\Sing(\cY_\eta)$ as desired.
\end{proof}
\begin{remark}
There are two different ways to view the $45$ quartics $f_{ij}$ appearing in the proof of \autoref{thm:coble-quartic}. Indeed, they can be seen as the $2\times2$-minors of the Jacobian matrix \eqref{eq:jacobian-matrix} for the quartic hypersurface $Z(q,f_1)$ in $Q$, but also as quartic equations obtained from the natural multiplication map
\[
m\colon H^0(\cO_{\bP}(1))\otimes \langle \tfrac{\partial q^2}{\partial x_0},\dots,\tfrac{\partial q^2}{\partial x_9} \rangle \to H^0(\cO_{\bP}(4))
\]
by considering, for each linear relation $\tfrac{\partial q}{\partial x_i}\otimes\tfrac{\partial q^2}{\partial x_j}- \tfrac{\partial q}{\partial x_j}\otimes\tfrac{\partial q^2}{\partial x_i}$ in $\ker(m)$,
the formal replacement of $q^2$ by~$f_1$ to define the quartic polynomial
$\tfrac{\partial q}{\partial x_i}\cdot\tfrac{\partial f_1}{\partial x_j}- \tfrac{\partial q}{\partial x_j}\cdot\tfrac{\partial f_1}{\partial x_i}$.
In the case of square $6$, working with minors will no longer be convenient due to the increased codimension of the intermediate subvariety $\Gr(2,6)\subset \bP^{14}$.
The second viewpoint of syzygies, although it may seem less natural, is better suited and will be adopted in the proof of \autoref{thm:coble-cubic} below.
\end{remark}

\section{Proof of \autoref{thm:coble-cubic}} \label{sec:thmB}
In this section we turn our attention to the case of square $6$ and the existence of the Coble type cubic. This case more or less follows the same line of reasoning as in square $4$, the key point being a specialization to the Hilbert square of a very general polarized K3 surface of genus $8$, whose geometry will be briefly reviewed at the beginning. 
Nevertheless, a few extra complications appear, and will be circumvented along Sections \ref{subsec:cohomology} and \ref{subsec:smoothness}.

\subsection{Hilbert squares of K3 surfaces of genus \texorpdfstring{$8$}{8}}
\label{subsec:genus-8}
Let $(S,L)$ be a polarized K3 surface of genus $8$ with $\Pic(S)=\bZ\cdot L$. The line bundle $L_2-2\delta$ on $S^{[2]}$ is very ample (\cite[Theorem~A]{projective-models}) of Beauville--Bogomolov square $6$ and divisibility $1$. The corresponding embedding
\begin{equation}\label{eq:S2-genus8}
    S^{[2]}\into \bP^{14}=\bP H^0(L_2-2\delta)^\vee
\end{equation}
has a rich geometry that we briefly review, following \cite[Section 5]{projective-models}.

Fix a $6$-dimensional vector space $V_6$, and consider the Grassmannian $\Gr(2,V_6^\vee)\subset \bP(\bw2 V_6^\vee)$ in its Plücker embedding.
Mukai \cite{mukai-models} proved that the embedding 
\[
S\into \bP(H^0(L)^\vee)=\bP^8
\]
is a $\bP^8$-linear section of $\Gr(2,V_6^\vee)$, under which the universal bundle of $\Gr(2,V_6^\vee)$ pulls back to $E^\vee$, where $E$ is the rigid stable vector bundle on $S$ with Mukai vector $v(E)=(2,L,4)$. In particular $V_6\cong H^0(S,E)$ canonically, and $(S,L)$ is determined by the $9$-dimensional vector subspace
$H^0(L)^\vee\into \bw2 H^0(E)^\vee$.
We denote by $W$ the $6$-dimensional quotient
\[
W\coloneqq\bw2 H^0(E)^\vee/H^0(L)^\vee.
\]
It turns out that there is a canonical identification $H^0(L_2-2\delta)\cong \bw2 V_6^\vee$, under which \eqref{eq:S2-genus8} is identified with a closed immersion
\[
S^{[2]}\into \bP(\bw2 V_6)
\]
which factors through $G\coloneqq\Gr(2,V_6)$. The immersion $S^{[2]}\into G$ admits two realizations:
\begin{itemize}
    \item It sends a length-$2$ subscheme $\xi\in S^{[2]}$ to the $2$-dimensional subspace 
    \[
    H^0(S,E\otimes \cI_\xi)\subset H^0(S,E)=V_6.
    \]
    In particular, the restriction $\cQ|_{S^{[2]}}$ of the universal quotient bundle $\cQ$ on $G$ equals the tautological bundle $E^{[2]}$ defined by $E$ on $S^{[2]}$.

    \item Considering the morphism of rank-$6$ vector bundles
    \[
    \varphi\colon \bw2\cQ^\vee \into \bw2 V_6^\vee \otimes \cO_G\onto W\otimes \cO_G=\cO_G^{\oplus 6}
    \]
    on $G$, the embedded $S^{[2]}\subset G$ coincides (scheme-theoretically) with the rank-$4$ degeneracy locus $D_4(\varphi)$.
\end{itemize}

Again, the degeneracy locus realization of $S^{[2]}$ provides a good understanding for the equations of $S^{[2]}$ in $\bP^{14}$, via the Gulliksen--Negård complex resolving $\cI_{S^{[2]}/G}$ in $G$:

\begin{proposition}[{\cite[Section 5.3]{projective-models}}]\label{prop:eqgenus8}
For every $d\geq2$, the vanishings
\[
H^1(Q,\cI_{S^{[2]}/Q}(d))=0=H^1(\bP^{14},\cI_{S^{[2]}/\bP^{14}}(d))
\]
hold. In particular, the natural restriction map
\[
m_d\colon H^0(\cO_{\bP^{14}}(d))\to H^0(\cO_{S^{[2]}}(d))
\]
is surjective. Moreover, $\ker(m_2)=H^0(\cI_{X/\bP^{14}}(2))$ is given by the $15$-dimensional vector space $H^0(\cI_{X/G}(2))$ of quadratic equations for $G$.
\end{proposition}

Again, from now on we will assume that $(S,L)$ is sufficiently general, so that by a Bertini-type argument, $S^{[2]}$ is the singular locus of the divisor $D_5(\varphi)=Z(\det\varphi)\in|\cO_G(3)|$ in $G$.
We recall also that $G$ is the singular locus of the Pfaffian cubic hypersurface $\Pf\coloneqq Z(\pf)\subset \bP^{14}$.
Hence the two normal bundles $\cN_{S^{[2]}/G}$ and $\cN_{G/\bP}$ both satisfy the relation $\cN\cong \cN^\vee(3)$, which follows from the existence of a nowhere degenerate quadratic form by \autoref{prop:coble-quadratic-form}.

\subsection{Cohomology computation}
\label{subsec:cohomology}

We keep the same notations as in the previous subsection. We begin with a natural interpretation of the first-order deformations of the Hilbert square $S^{[2]}$ in $G$, which is the analogue of \autoref{prop:64genus7} in genus $8$.

\begin{proposition}\label{54again}
There is a canonical identification
\[
H^0(S^{[2]}, \cN_{S^{[2]}/G})\cong H^0(S,L)\otimes W.
\]
\end{proposition}

\begin{proof}
The proof is divided into two steps. First we will construct a canonical inclusion 
\[
H^0(L)\otimes W\into H^0(S^{[2]},\cN_{S^{[2]}/G}),
\]
and then show that $h^0(\cN_{S^{[2]}/G})=\dim(H^0(L)\otimes W)=54$.

\vspace{2mm}

\textbf{Step 1.}
Using the degeneracy locus description and writing $\cE_1\coloneqq\ker(\varphi|_{S^{[2]}})$ and $\cE_2\coloneqq\mathrm{coker}(\varphi|_{S^{[2]}})$,
we have the following description for the normal bundle
\[
\cN_{S^{[2]}/G}\cong \cE_1^\vee\otimes \cE_2.
\]
As explained in \cite[Section 5.2]{projective-models}, the kernel bundle
\[
\cE_1=\ker(\varphi|_{S^{[2]}})\into \bw2 \cQ^\vee|_{S^{[2]}}\into \bw2 V_6^\vee\otimes \cO_{S^{[2]}}
\]
is the universal secant line to $S\subset \bP(H^0(L)^\vee)\subset \bP(\bw2 V_6^\vee)$. In other words, $\cE_1^\vee$ is the tautological rank-$2$ bundle $L^{[2]}=\pi_*\psi^* L$ on $S^{[2]}$ induced by $L$,
where $\pi$ and $\psi$ are the maps given in the diagram
\[
        \begin{tikzcd}
            &B\ar[ld, "\pi"']\ar[rd, "\psi"]\coloneqq\setmid{(p,\xi)\in S\times S^{[2]}}{p\in \xi}\\
            S^{[2]}&& S
        \end{tikzcd}   
\]
Therefore, to compute the global sections of $\cN_{S^{[2]}/G}$ we have
\[
H^0(S^{[2]}, \cN_{S^{[2]}/G})=H^0(S^{[2]}, \pi_*\psi^*L\otimes \cE_2)=H^0(S, L\otimes \psi_*\pi^*\cE_2).
\]
We claim that the trivial vector bundle $W\otimes \cO_S$ is canonically contained in $\psi_*\pi^*\cE_2$. Verification of this claim is enough to conclude Step 1,  since we obtain the desired canonical inclusion
\[
H^0(L)\otimes W = H^0(S,L\otimes W\otimes \cO_S)  \into H^0(S, L\otimes \psi_*\pi^*\cE_2)=H^0(S^{[2]}, \cN_{S^{[2]}/G}).
\]

To verify the claim, consider the Pfaffian cubic fourfold $Y\subset \bP(W^\vee)=\bP^5$ obtained as a $\bP^5$-linear section of the Pfaffian cubic hypersurface $\Pf\subset \bP(\bw2 V_6)=\bP^{14}$. As explained in \cite[Remark 5.3]{projective-models}, under the isomorphism 
\[
S^{[2]}\cong F(Y)\subset \Gr(2,W^\vee)
\]
proven by Beauville--Donagi \cite{Beauville-Donagi}, $\cE_2$ is (the pullback of) the dual of the universal rank-$2$ vector bundle $\cU_W$ on $\Gr(2,W^\vee)$.
Let $I\coloneqq\bP_{F(Y)}(\cU_W)$ be the natural incidence between $Y$ and $F(Y)$ (that is, the universal line) with projections $\psi'$ and $\pi'$, which gives the diagram
\[
\begin{tikzcd}
& I\ar[ddl,"\psi'"']\ar[ddrr,"\pi'"] &&&& B\ar[ddll,"\pi"']\ar[ddr,"\psi"]\\
&&& \\
Y &  && F(Y)\cong S^{[2]} &&& S
\end{tikzcd}
\]
Denoting by $\zeta$ the relative $\cO(1)$ on the projective bundle $I$, we have $(\psi')^*\cO_Y(1)=\zeta$ and $\pi'_*\zeta=\cU_W^\vee=\cE_2$. Therefore, the natural map
\[
W\otimes \cO_Y=H^0(\cO_Y(1))\otimes \cO_Y\to \cO_Y(1)
\]
induces $W\otimes \cO_{S^{[2]}}\to \cE_2$, which in turn induces a canonical map
\[
W\otimes \cO_{S}\to \psi_*\pi^*\cE_2.
\]

The latter map is the one whose injectivity is claimed. By cohomology and base change, it suffices to check that for a general point $p\in S$ the restricted map
\[
W\to H^0(\cE_2|_{\tS_p})
\]
is an inclusion, where $\tS_p\coloneqq\psi^{-1}(p) \cong \Bl_p(S)$ is the subvariety of $S^{[2]}$ parametrizing subschemes with support containing $p$. Such inclusion is readily seen to be equivalent to the non-degeneracy in $\bP(W^\vee)=\bP^5$ of the subvariety $\Sigma_p\subset Y$ swept out by the surface of lines $\tS_p\subset S^{[2]}\cong F(Y)$.

Since $\tS_p$ contains the smooth rational curve $\Delta_p\coloneqq \bP T_p S$, it suffices to show the non-degeneracy of the corresponding quintic scroll.
By \cite[Section~4.1.3]{Hassett}, a general such quintic scroll is a divisor of type $(3,1)$ on a Segre-embedded $\bP^1\times \bP^2$ in $\bP^5$, so it cannot be degenerate.

\vspace{2mm}

\textbf{Step 2.} Now we will prove that $h^0(\cN_{S^{[2]}/G})=54$. We begin by showing the equality $h^0(S^{[2]},\cT_G|_{S^{[2]}})=35$. To this end, observe that $\cT_G=\cU^\vee\otimes \cQ$, where $\cU$ and $\cQ$ are the universal and quotient bundles fitting in the short exact sequence
\[
0\to \cU\to V_6\otimes \cO_G\to \cQ\to 0.
\]
Recall also that the restriction $\cQ|_{S^{[2]}}$ is nothing but the tautological bundle $E^{[2]}=\pi_*\psi^*E$ defined by $E$ on $S^{[2]}$, where $E$ is the stable bundle on $S$ with Mukai vector $v(E)=(2,L,4)$.
The equality $h^0(\cT_G|_{S^{[2]}})=35$ is then obtained by combining the short exact sequence
\[
0\to \cQ^\vee|_{S^{[2]}}\otimes \cQ|_{S^{[2]}}\to H^0(E)^\vee\otimes \cQ|_{S^{[2]}}\to \cT_G|_{S^{[2]}}\to 0
\]
with the following cohomological calculations (for which we use \cite[Remark~3.19]{Krug:taut}):

\begin{itemize}
    \item $h^0(\cQ^\vee|_{S^{[2]}}\otimes \cQ|_{S^{[2]}})=\hom(E^{[2]},E^{[2]})=1$
    \item $h^0(H^0(E)^\vee\otimes \cQ|_{S^{[2]}})=6h^0(E^{[2]})=6h^0(E)=36$
    \item $h^1(\cQ^\vee|_{S^{[2]}}\otimes \cQ|_{S^{[2]}})=\mathrm{ext}^1(E^{[2]},E^{[2]})=0$
\end{itemize}

Now, in view of the equality $h^0(S^{[2]},\cT_G|_{S^{[2]}})=35$, the assertion $h^0(\cN_{S^{[2]}/G})=54$ is equivalent to the Kodaira--Spencer map
\[
H^0(S^{[2]},\cN_{S^{[2]}/G})\to H^1(S^{[2]},\cT_{S^{[2]}})
\]
being of rank $19$. In other words, it suffices to show that a general first-order deformation of $S^{[2]}$ in $\bP^{14}$ is not contained in $G$ up to a first-order linear transformation, or equivalently, a general member $(X,H)\in \cM_6^{(1)}$ is not contained in any Plücker embedded $G$.

To check this, it suffices to find a degeneration of $(X,H)\in\cM_6^{(1)}$ in $\bP^{14}$ which does not lie in $G$.
To this end, we consider the pair $(S'^{[2]},L'_2)$ for a general polarized K3 surface $(S',L')$ of genus $4$.
Note that the divisor $L'_2$ is big and basepoint-free (but not ample).
As will be shown in \autoref{lem:D6-square6}, the corresponding morphism
\[
S'^{[2]}\to \bP H^0(S'^{[2]},L'_2)^\vee
\]
is the Hilbert--Chow contraction $S'^{[2]}\to \Sym^2S'$ followed by a closed immersion, and the image $\Sym^2S'$ is indeed not contained in any Plücker embedded Grassmannian $\Gr(2,6)$.
Thus the result follows.
\end{proof}

It follows from the proposition that  $H^0(\cN_{S^{[2]}/Q})$ is canonically identified with the Zariski tangent space to the Grassmannian $\Gr(9, \bw2 H^0(E)^\vee)$ at the point $[H^0(L)^\vee]$. That is, a first-order deformation of $S^{[2]}$ in $G$ is the same as a first-order deformation of the chosen $\bP^8$ (determining the K3 surface $S$) in the ambient projective space of the Grassmannian $\Gr(2,V_6^\vee)$.
Moreover, as explained in the proof, we obtain the following conclusion.

\begin{corollary}\label{prop:deformation-genus8}
A general member $(X,H)\in\cM_6^{(1)}$ in $\bP^{14}$ is not contained in any Plücker embedded Grassmannian $G=\Gr(2,V_6)$.
\end{corollary}

The following lemma completes the remaining step in the proof of \autoref{54again}.
We briefly digress from the case of K3 surfaces of genus $8$ and consider instead those of genus $4$.

\begin{lemma}\label{lem:D6-square6}
Consider $(S^{[2]},L_2)$ for $(S,L)$ a general polarized K3 surface of genus $4$, namely, where $S\into \bP^4=\bP H^0(L)^\vee$ is a $(2,3)$-complete intersection.
The morphism
\[
\varphi_{L_2}\colon S^{[2]}\to \bP H^0(S^{[2]},L_2)^\vee
\]
is the Hilbert--Chow contraction $S^{[2]}\to \Sym^2S$ followed by a closed immersion, and the image $\Sym^2S$ is not contained in any Plücker embedded Grassmannian $\Gr(2,6)$.
\end{lemma}
\begin{proof}
We will denote by $Q=Z(q)$ the quadric containing $S$, which we assume to be the standard one $q=x_0^2+\dots+x_4^2$.
The corresponding morphism $S^{[2]}\to \bP(H^0(L_2)^\vee)=\bP(\Sym^2 H^0(L)^\vee)$ is the composition
\[
S^{[2]}\overset{\mathrm{HC}}{\to}\Sym^2 S\into \Sym^2 Q\into \Sym^2 \bP^4\into \bP(\Sym^2 H^0(L)^\vee)=\bP^{14}.
\]
Factoring the restriction $H^0(\bP^{14}, \cO(d))\to H^0(\Sym^2S, \cO(d))$ through $\Sym^2 \bP^4$ we obtain
\[
m_d\colon\Sym^d\Sym^2 H^0(L)\overset{m'_d}{\to} \Sym^2\Sym^d H^0(L) \xrightarrow{\Sym^2 r_d} \Sym^2 H^0(L^{\otimes d}),
\]
where the first arrow is the surjective map given by
\[
(x_{i_1}\otimes x_{j_1})\otimes\dots\otimes(x_{i_d}\otimes x_{j_d})\mapsto (x_{i_1}\otimes\dots\otimes x_{i_d})\otimes (x_{j_1}\otimes\dots\otimes x_{j_d}),
\]
and the second arrow is the symmetric square of the restriction $r_d\colon \Sym^d H^0(L) \to H^0(L^{\otimes d})$ on $S\subset\bP^4$.
In particular, by the projective normality of $S$ in $\bP^4$, $\Sym^2 S$ is projectively normal in $\bP^{14}$, and is contained in $15$ quadrics by a dimension count.

On the other hand, we can explicitly exhibit $15$ quadrics in $\bP^{14}$ containing $\Sym^2 Q$ and in turn $\Sym^2 S$.
Namely, the map $m_2'$ is an isomorphism, while the kernel of $\Sym^2 r_2$ is spanned by $q\otimes q' + q'\otimes q$ for $q'\in H^0(\bP^4,\cO(2))$.
The $15$ quadrics are thus obtained as their preimage under $m_2'$ and can be explicitly computed.
In particular, they depend only on $Q$ and not on the K3 surface $S$.

Finally, if $\Sym^2S$ were to be contained in some Plücker embedded $G$, then the $15$ quadrics we obtained would be precisely the $15$ Plücker quadrics up to some linear transformation.
Hence to disprove this, it suffices to show that the intersection of the $15$ quadrics is not $G$.%
\footnote{The intersection of the $15$ quadrics in fact coincides with $\Sym^2 Q$. This fact can be easily checked via {\sc Macaulay2}.}
For this, we pick an explicit point $p\coloneqq \set{x,y}$ on $\Sym^2 Q$ given by $x\coloneqq(1:i:0:0:0)$ and $y\coloneqq(1:-i:0:0:0)$, and compute the Jacobian matrix of the $15$ quadrics at this point: we obtain a matrix of rank $8$, so the intersection of the quadrics has dimension $\le 6$ at $p$ and cannot be $G$.
This concludes the proof.
\end{proof}

Now we go back to study $(S^{[2]}, L_2-2\delta)$ for $(S,L)$ a general K3 surface of genus $8$ and perform some further cohomology computation.

\begin{lemma}
\label{lem:190}
Consider $(S^{[2]},L_2-2\delta)$ for $(S,L)$ a general polarized K3 surface of genus $8$, and let $G\coloneqq\Gr(2,6)\subset \bP^{14}=\bP(H^0(S^{[2]},L_2-2\delta)^\vee)$.
We have $h^i(G,\cN_{G/\bP}\otimes \cI_{S^{[2]}/G})=0$ for all $i$, except for $h^1=1$,
and $h^0(S^{[2]},\cN_{G/\bP})=190$ while all higher cohomologies vanish.
\end{lemma}
We note that the analogue of this lemma in square $4$ is computing $h^0(S^{[2]},\cO(2))=55$, which required no extra work.
\begin{proof}
Denote by $\Sigma^\lambda$ the Schur functor for a partition $\lambda$.
For $\cN_{G/\bP}\otimes\cI_{S^{[2]}/G}$, we consider the Gulliksen--Negård complex $F_\bullet$ (see \cite[Section~5.3]{projective-models})
\[
0\to \cO_G(-6)\to M(-3)^{\oplus 6}\to \Lambda M\oplus \cO_G(-3)^{\oplus 35} \to M(-2)^{\oplus 6} \to \cI_{S^{[2]}/G},
\]
where $M$ is the vector bundle $\cN_{G/\bP}(-2)=\Sigma^{1,1}\cQ^\vee$ and $\Lambda M$ is the rank-$35$ vector bundle $\Sigma^{2,1,1,1,1}M$ that fits in the short exact sequence
\[
0\to \Lambda M \to M\otimes \bw5 M\to \det M\to 0.
\]
One checks that $\Lambda M$ is the direct sum of two irreducible homogeneous vector bundles $\Sigma^{2,2}\cQ^\vee(-2)$ and $\Sigma^{2,1,1}\cQ^\vee(-2)$.
We then tensor the Gulliksen--Negård complex by $\cN_{G/\bP}$ and compute the associated spectral sequence
\[
E_1^{-k,i+k}\coloneqq H^{i+k}(G,\cN_{G/\bP}\otimes F_k)\Longrightarrow H^i(G,\cN_{G/\bP}\otimes\cI_{S^{[2]}/G}).
\]
Using Borel--Weil--Bott, we check that almost all terms are zero, with the only nonzero term arises when we compute
\[
\begin{gathered}
\cN_{G/\bP} \otimes \Lambda M = \Sigma^{1,1}\cQ^\vee\otimes (\Sigma^{2,2}\cQ^\vee\oplus\Sigma^{2,1,1}\cQ^\vee)\\
=\Sigma^{3,3}\cQ^\vee\oplus2\Sigma^{3,2,1}\cQ^\vee\oplus\Sigma^{3,1,1,1}\cQ^\vee\oplus\Sigma^{2,2,2}\cQ^\vee\oplus2\Sigma^{2,2,1,1}\cQ^\vee,
\end{gathered}
\]
where $h^2(\Sigma^{3,1,1,1}\cQ^\vee)=1$ so $\dim E_1^{-1,2}=1$.
We conclude that $h^i(\cN_{G/\bP}\otimes\cI_{S^{[2]}/G})=0$ for all $i$, except for $h^1=1$.

Then we consider the following short exact sequence
\[
0\to \cN_{G/\bP}\otimes\cI_{S^{[2]}/G}\to \cN_{G/\bP}\to \cN_{G/\bP}|_{S^{[2]}}\to 0.
\]
The numbers $h^i(G, \cN_{G/\bP})=0$ are readily known, with $h^0=189$ and $h^i=0$ for all $i>0$.
This gives us $h^0(S^{[2]},\cN_{G/\bP})=190$ while $h^i(S^{[2]},\cN_{G/\bP})=0$ for all $i>0$, as desired.
\end{proof}

Our next preparatory result is the analogue of \autoref{prep-lemma}:

\begin{lemma}\label{prep-lemma-cubic}
    Consider $(S^{[2]},L_2-2\delta)$ for $(S,L)$ a general polarized K3 surface of genus $8$, and let $G\coloneqq\Gr(2,6)\subset \bP^{14}=\bP(H^0(S^{[2]},L_2-2\delta)^\vee)$. Then:
    \begin{enumerate}
        \item\label{prep-lemma-cubic-1} $H^0(G,\cI_{S^{[2]}/G}^2(3))\cong\bC$, and is generated by the first degeneracy locus $D_5(\varphi)\in|\cO_G(3)|$.
        Furthermore, we have the vanishing $H^1(G,\cI_{S^{[2]}/G}^2(3))=0$. 
        \item\label{prep-lemma-cubic-2} $H^0(\bP^{14},\cI_{S^{[2]}/\bP}^2(3))\cong\bC$, and is generated by the Pfaffian cubic $\Pf\in |\cO_{\bP}(3)|$.
        Furthermore, we have the vanishing $H^1(\bP^{14},\cI_{S^{[2]}/\bP}^2(3))=0$. 
    \end{enumerate}
\end{lemma}
\begin{proof}
\eqref{prep-lemma-cubic-1} follows from \cite[Remark 5.6]{projective-models} by observing that, under the isomorphisms $H^0(L)\otimes W\cong H^0(\cN_{S^{[2]}/G})\cong H^0(\cN_{S^{[2]}/G}^\vee(3))$ provided by \autoref{54again}, the surjection
\[
H^0(G,\cI_{S^{[2]}/G}(3))\onto H^0(L)\otimes W
\]
corresponds to the natural map $\cI_{S^{[2]}/G}(3)\onto \cN_{S^{[2]}/G}^\vee(3)$.

Now we prove \eqref{prep-lemma-2}.
For a cubic hypersurface $F\coloneqq Z(f)\in |\cI^2_{S^{[2]}/\bP}(3)|$ that is singular along $S^{[2]}$, its partial derivatives $\frac{\partial f}{\partial x_i}$ must all lie in $H^0(\cI_{S^{[2]}/\bP}(2))$, which coincides with the vector space of Plücker quadrics $H^0(\cI_{G/\bP}(2))$.
Therefore, $F$ is singular along $G$ as well and so it must be the Pfaffian cubic.
Then we examine the following commutative diagram with exact rows and columns
\[
\begin{tikzcd}
 &  0\ar[d] & 0\ar[d] & 0\ar[d] & \\
 0\ar[r] & \cK\ar[r]\ar[d] & \cI_{G/\bP}(3)\ar[r]\ar[d] & \cN_{G/\bP}^\vee(3)|_{S^{[2]}} \ar[r]\ar[d] & 0\\
0\ar[r] & \cI_{S^{[2]}/\bP}^2(3)\ar[r]\ar[d] & \cI_{S^{[2]}/\bP}(3)\ar[r]\ar[d] & \cN_{S^{[2]}/\bP}^\vee(3) \ar[r]\ar[d] & 0\\
0\ar[r] & \cI_{S^{[2]}/G}^2(3)\ar[r]\ar[d]& \cI_{S^{[2]}/G}(3)\ar[d]\ar[r] & \cN_{S^{[2]}/G}^\vee(3)\ar[d]\ar[r] & 0\\
 & 0 & 0 & 0 & 
\end{tikzcd}
\]
where $\cK$ sits as an extension
\[
0\to \cI^2_{G/\bP}(3)\to \cK\to \cN^\vee_{G/\bP}(3)\otimes\cI_{S^{[2]}/G}\cong \cN_{G/\bP}\otimes \cI_{S^{[2]}/G}\to 0.
\]
We note that $h^i(\cI^2_{G/\bP}(3))=0$ for $i\ge 1$ since $h^i(\cI_{G/\bP}(3))$ and $h^i(\cN_{G/\bP})$ are all known.
So by \autoref{lem:190}, we obtain $h^0(\cK) =h^0(\cI^2_{G/\bP}(3))=1$ and  $h^1(\cK)=h^1(\cN_{G/\bP}\otimes \cI_{S^{[2]}/G})=1$.
Looking at cohomology in the first column of the diagram, we see that
$H^0(\cI^2_{G/\bP}(3))=H^0(\cK)=H^0(\cI^2_{S^{[2]}/\bP}(3))$, all being generated by the Pfaffian $\pf$.
Therefore, we have another isomorphism $H^0(\cI^2_{S^{[2]}/G}(3))\simto H^1(\cK)=H^1(\cN_{G/\bP}\otimes \cI_{S^{[2]}/G})$ given by the boundary map, from which we may then conclude that $H^1(\bP^{14},\cI_{S^{[2]}/\bP}^2(3))=0$ since $H^1(G,\cI_{S^{[2]}/G}^2(3))=0$.
\end{proof}

\subsection{Smoothness criterion}\label{subsec:smoothness}
Unlike the case of square $4$, due to the large codimension of the Grassmannian $G$ in $\bP^{14}$, using the Jacobian criterion to characterize the singular locus of a cubic hypersurface in $G$ is not feasible.
The main goal of this section is to establish an alternative criterion in \autoref{prop:35cubics}.

Recall that $G$ is the singular locus of the Pfaffian cubic hypersurface $\Pf=Z(\pf)$, and the $15$ Plücker quadrics can be identified as the partial derivatives $\frac{\partial \pf}{\partial x_i}$ for $\set{x_i}_{0\le i \le 14}$.

We begin by a simple lemma.
\begin{lemma}
\label{lemma:xi_f_as_partial_pf}
Consider the Pfaffian cubic $\pf$.
For each $x_i$ there exists some $15\times 15$ matrix $A_i$ with scalar entries such that the quartic $x_i\cdot \pf$ can be written as
\[
x_i\cdot \pf =
\partial \pf \cdot A_i \cdot \partial \pf^T,
\]
where $\partial\pf$ denotes the row matrix $\begin{psmallmatrix} \frac{\partial \pf}{\partial x_{0}}&\dots&\frac{\partial \pf}{\partial x_{14}} \end{psmallmatrix}$.
\end{lemma}
\begin{proof}
We work in the Plücker coordinates with $\set{x_0,\dots,x_{14}}=\set{p_{01},\dots,p_{45}}$ so the Pfaffian cubic becomes
\[
\pf = \sum_{(ab)(cd)(e\!f)\vdash\set{0,\dots,5}}(-1)^{\sigma(abcde\!f)}p_{ab}p_{cd}p_{e\!f},
\]
where the sum is taken over all $15$ partitions of $\set{0,\dots,5}$ into three pairs, and $\sigma$ is the signature of a permutation.
We then check directly that $A_i$ can be taken as the matrix of the Plücker quadric $\frac{\partial\pf}{\partial x_i}$.
For example, for $x_0=p_{01}$ we have
\[
p_{01}\cdot\pf=
\tfrac{\partial \pf}{\partial p_{23}}\tfrac{\partial \pf}{\partial p_{45}}-
\tfrac{\partial \pf}{\partial p_{24}}\tfrac{\partial \pf}{\partial p_{35}}+
\tfrac{\partial \pf}{\partial p_{25}}\tfrac{\partial \pf}{\partial p_{34}},
\]
using the matrix of the quadric
$\tfrac{\partial \pf}{\partial p_{01}}=p_{23}p_{45}-p_{24}p_{35}+p_{25}p_{34}$.
\end{proof}

Recall that the $35$-dimensional kernel of the multiplication map 
\[
H^0(\cO_{\bP}(1))\otimes H^0(\cI_{G/\bP}(2))\onto H^0(\cI_{G/\bP}(3)),
\]
parametrizes linear relations among the quadrics $\frac{\partial \pf}{\partial x_{0}},\dots,\frac{\partial \pf}{\partial x_{14}}$.
In the sequel, we fix a basis
\[
    \set{\sum_{i=0}^{14} \ell_{i}^{(m)}\otimes \frac{\partial \pf}{\partial x_{i}}}_{m\in\set{1,\dots,35}}
\]
for the kernel, where $\ell_i^{(m)}$ are linear forms.
We will also denote by $\Syz$ the $15\times 35$ matrix with entries being $\ell_i^{(m)}$ so that
\[
\partial \pf\cdot \Syz = 0,
\]
where $\partial\pf$ is the row matrix $\begin{psmallmatrix} \frac{\partial \pf}{\partial x_{0}}&\dots&\frac{\partial \pf}{\partial x_{14}} \end{psmallmatrix}$.

Our next observation is that these linear relations provide a criterion for determining the singular locus of a cubic hypersurface in $G$ more convenient than the Jacobian criterion, due to the large codimension of $G$ in $\bP^{14}$.
\begin{proposition}\label{prop:35cubics}
Let $F\coloneqq Z(f)$ be a cubic hypersurface defined by $f\in H^0(\cO_{\bP}(3))$ not containing the Grassmannian $G$.
We consider the following $35$ cubic polynomials
\[
f^{(m)}\coloneqq\sum_{i=0}^{14} \ell_{i}^{(m)}\cdot \frac{\partial f}{\partial x_{i}},\quad
m\in\set{1,\dots,35}.
\]
Then the singular locus $X\coloneqq\Sing(G\cap F)$ contains the scheme-theoretical intersection $X_0\coloneqq G\cap Z(f^{(m)})$ of $G$ and the $35$ cubics $Z(f^{(m)})$.
Conversely, we also have the inclusion $X_\red\subset X_0$.
\end{proposition}
Namely, in our case, when $F=Z(f)$ is given by the first degeneracy locus $D_5(\varphi)$, we see that $S^{[2]}=\Sing(F)$ coincides with the scheme-theoretical intersection of $G$ and the $35$ cubics $Z(f^{(m)})$.

Note also that in the criterion we do not require the original cubic hypersurface $F$, as its equation can be recovered from the $35$ cubics.
This is similar to when one applies the usual Jacobian criterion for a hypersurface, where the hypersurface itself can be recovered from the partial derivatives using Euler's identity.
\begin{proof}
First we prove that $f^{(m)}(x)=0$ for $x\in X=\Sing(G\cap F)$, which will show the inclusion $X_\red\subset X_0$.
Recall that $G$ is the singular locus of the Pfaffian cubic hypersurface, and the $15$ Plücker quadrics can be identified as the partial derivatives $\frac{\partial \pf}{\partial x_i}$.
In particular, the Hessian matrix of the Pfaffian cubic
\[
H\coloneqq \left(\dfrac{\partial^2\pf}{\partial x_i\partial x_j}\right)_{0\le i,j\le 14}
\]
has rank $6$ everywhere on $G$.

Differentiating each linear relation $\sum_{i=0}^{14} \ell_{i}^{(m)}\cdot \frac{\partial \pf}{\partial x_{i}}=0$ with respect to the variable $x_{j}$ gives
\[
\sum_{i=0}^{14} \Big(\ell_{i}^{(m)}\cdot \frac{\partial^2 \pf}{\partial x_{j} \partial x_{i}}+\frac{\partial \ell_{i}^{(m)}}{\partial x_{j}}\cdot \frac{\partial \pf}{\partial x_{i}}\Big)=0.
\]
Since $\frac{\partial \pf}{\partial x_i}$ vanishes along $G$, this shows that the following quadrics vanishes along $G$ as well
\begin{equation}
\label{eq:H_Syz_vanishes}
\sum_{i=0}^{14} \ell_{i}^{(m)}\cdot \frac{\partial^2 \pf}{\partial x_{j} \partial x_{i}}\in H^0(\cI_{G/\bP}(2)).
\end{equation}
In other words, the matrix $(H\cdot \Syz)|_G$ is identically zero, where $\Syz$ is the matrix with linear entries $\ell_i^{(m)}$.

On the other hand, consider the $16\times 15$ Jacobian matrix of $G\cap F$
\[
J\coloneqq 
\begin{pNiceArray}{c}
H\\
\hdottedline
\partial f
\end{pNiceArray}
\]
with an extra row of the partial derivatives of $f$.
By the Jacobian criterion, the matrix $J$ must have rank at most $6$ for $x\in \Sing(G\cap F)$, so the last row must be a linear combination of the other rows.
In other words, there are scalars $\gamma_{0},\dots,\gamma_{14}\in\bC$ such that
\[
\forall 0\le i\le 14,\quad \frac{\partial f}{\partial x_{i}}(x)=\sum_{j=0}^{14}\gamma_{j}\cdot \frac{\partial ^2\pf}{\partial x_{i}\partial x_{j}}(x),
\]
which gives 
\[
f^{(m)}(x)=\sum_{j=0}^{14}\gamma_{j}\Big(\sum_{i=0}^{14}\ell_{i}^{(m)}(x)\cdot \frac{\partial^2\pf}{\partial x_{i} \partial x_{j}}(x)\Big)=0
\]
by~\eqref{eq:H_Syz_vanishes}, as desired.

Now we prove the other inclusion $X_0\subset X$.
Note that since $X$ is the singular locus $\Sing(G\cap F)$, it is the scheme-theoretical intersection of $G$, $F$, and all the $7\times 7$-minors of the Jacobian matrix $J$.
Therefore, it suffices to show that $X_0$ is contained in $F$ and all $7\times 7$-minors of $J$ vanish along $X_0$, that is, $\bw7 J$ is identically zero along $X_0$.
We first prove the latter.

We have the following exact sequence
\[
\dots \to \cO_{\bP}(-3)^{\oplus 35} \stackrel{\Syz}{\to} \cO_{\bP}(-2)^{\oplus 15}\to \cI_{G/\bP}\to 0
\]
where $\Syz$ is defined by the linear forms $\ell_i^{(m)}$. Tensoring with $\cO_G$ yields an exact sequence
\[
\cO_{G}(-3)^{\oplus 35} \stackrel{\Syz}{\to} \cO_{G}(-2)^{\oplus 15}\to
\cN_{G/\bP}^\vee\to 0
\]
of vector bundles on $G$, and dualizing we obtain the exact sequence
\[
0 \to \cN_{G/\bP} \to \cO_G(2)^{\oplus 15} \xrightarrow{\Syz^T|_G} \cO_G(3)^{\oplus 35}
\]
that identifies $\cN_{G/\bP}$ as the kernel bundle of $\Syz^T|_G$.

On the other hand, we observe that
\[
J\cdot \Syz =
\begin{pNiceArray}{ccc}
\Block{2-3}{H\cdot \Syz}\\
&\\
\hdottedline
f^{(1)}&\cdots&f^{(35)}
\end{pNiceArray},
\]
where we recall that the matrix $H\cdot \Syz$ becomes identically zero along $G$, while the last row consists of exactly the $35$ cubics $f^{(m)}$.
Therefore, the matrix $J\cdot \Syz$ becomes identically zero along $X_0$, and the image of $J^T|_{X_0}$ as the vector bundle map
\[
J^T|_{X_0}\colon \cO_{X_0}(1)^{\oplus 15}\oplus \cO_{X_0}\to \cO_{X_0}(2)^{\oplus 15}
\]
lies in the kernel of $\Syz^T|_{X_0}$, which is exactly $\cN_{G/\bP}|_{X_0}$.
In other words, $J^T|_{X_0}$ factors as
\[
J^T|_{X_0}\colon \cO_{X_0}(1)^{\oplus 15}\oplus \cO_{X_0}\to \cN_{G/\bP}|_{X_0} \into \cO_{X_0}(2)^{\oplus 15}.
\]
This means $\bw 7 J^T$ factors through $\bw7\left(\cN_{G/\bP}|_{X_0}\right)=0$ and therefore must be identically zero along $X_0$.

It remains to show that $X_0$ is contained in the cubic hypersurface $F$.
For each $x_i$, by \autoref{lemma:xi_f_as_partial_pf} we have a matrix $A_i$ such that
\[
x_i\cdot \pf =\partial \pf \cdot A_i\cdot \partial \pf^T.
\]
On the other hand, Euler's identity for homogeneous polynomials implies that
\[
x_i\cdot \pf = x_i\cdot (\tfrac13 \partial \pf \cdot x^T),
\]
where $x^T$ is the column vector consisting of the variables $x_i$.
Combining the two, we see that
\[
\partial \pf \cdot (A_i\cdot \partial \pf^T - \tfrac13 x_i\cdot x^T)=0,
\]
which provides a quadratic relation among $\partial \pf$.
This means that there exists a $35\times 1$ row matrix $B_i$ with linear entries such that
\[
A_i\cdot \partial \pf^T - \tfrac13 x_i\cdot x^T = \Syz\cdot B_i.
\]
Now we consider
\[
\begin{aligned}
x_i\cdot f &= x_i\cdot (\tfrac13 \partial f \cdot x^T) \quad \text{(Euler's identity)}\\
&= \partial f\cdot(\tfrac13 x_i\cdot x^T)\\
&=\partial f \cdot (A_i\cdot \partial \pf^T - \Syz\cdot B_i)\\
&=\partial f \cdot A_i\cdot \partial \pf^T - (f^{(m)})_{1\le m\le 35}\cdot B_i\quad  \text{(definition of the cubics $f^{(m)}$)}.
\end{aligned}
\]
Thus $x_i\cdot f$ vanishes whenever $\partial \pf$ and $f^{(m)}$ vanish, in other words, it vanishes along $X_0=G\cap Z(f^{(m)})_{1\le m\le 35}$.
Therefore $f$ must vanish along $X_0$ as well, which concludes the proof.
\end{proof}

\subsection{End of proof}
With all these ingredients, we are now in a position to prove \autoref{thm:coble-cubic}.

\begin{proof}[Proof of \autoref{thm:coble-cubic}]
Consider a general $(X,H)\in\cM_6^{(1)}$. We want to prove the existence of a unique cubic hypersurface $Y\subset \bP^{14}=\bP(H^0(X,H)^\vee)$ such that $X=\Sing(Y)$ scheme-theoretically.
We will proceed in two steps just as in the case of square $4$.

First we prove that $h^0(\bP^{14},\cI_{X/\bP}^2(3))=1$.
We consider the short exact sequence
\[
0\to \cI_{X/\bP}^2(3)\to \cI_{X/\bP}(3)\to \cN_{X/\bP}^\vee(3)\to 0.
\]
Cubic normality of $X$ in $\bP^{14}$ tells us that $h^0(\cI_{X/\bP}(3))=245$, $h^1(\cI_{X/\bP}(3))=0$, and thus
\begin{equation}
\label{eq:245}
h^0(\bP^{14},\cI_{X/\bP}^2(3)) + h^0(X,\cN_{X/\bP}^\vee(3)) = 245 + h^1(\bP^{14},\cI_{X/\bP}^2(3)).
\end{equation}
Now we let $(X,H)$ specialize to $(S^{[2]},L_2-2\delta)$ for $(S,L)$ a general K3 surface of genus $8$.
In this case, \eqref{eq:245} still holds since $S^{[2]}$ is cubically normal in $\bP^{14}$ by \autoref{prop:eqgenus8}, while $h^0(\bP^{14},\cI_{S^{[2]}/\bP}^2(3))=1$ and $h^1(\bP^{14},\cI_{S^{[2]}/\bP}^2(3))=0$ by \autoref{prep-lemma-cubic}.(\ref{prep-lemma-cubic-2}).
Using upper semicontinuity of cohomology ranks and the relation \eqref{eq:245}, these equalities must hold for a general $(X,H)\in\cM_6^{(1)}$ as well.

Next we show that $Y$ is singular \emph{exactly} along $X$.
Consider a general one-parameter flat family
\[
\begin{tikzcd}
\cX\ar[d]\subset \bP^{14}\times\Delta\\
\Delta
\end{tikzcd}
\]
with central fiber $\cX_0=(S^{[2]}, L_2-2\delta)$, where $S$ is a general K3 surface of genus~$8$.
Again we assume that the base $\Delta$ is the formal scheme $\Spec\bC[[t]]$ and denote by $\eta$ the generic point.
Recall from \autoref{prep-lemma-cubic}.(\ref{prep-lemma-cubic-2}) that the Pfaffian cubic $Z(\pf)$ is the unique cubic hypersurface in $\bP^{14}$ that is singular along $\cX_0$.
We obtain a cubic hypersurface $\cY\coloneqq Z(F)\subset \bP^{14}\times\Delta$, where
\[
F\coloneqq \pf+t\cdot f_1+t^2\cdot f_2 +\dots\in H^0(\bP^{14},\cO(3))\otimes_\bC\bC[[t]],
\]
such that $\cY_\eta$ is the unique cubic hypersurface in $\bP^{14}_{k(\eta)}$ whose singular locus contains $\cX_\eta$.
Note that, by the generality assumption on the family, $f_1$ does not vanish entirely along the Grassmannian $G$.
Otherwise, one may write $f_1=\sum_{i=0}^{14}\ell_i\cdot \frac{\partial\pf}{\partial x_i}$, where $\ell_i=\ell_i(x)$ are linear forms.
Then we may check that in first-order we have
\[
\pf + \varepsilon \cdot f_1 = \pf(x+\varepsilon \cdot \ell(x)).
\]
Therefore, up to taking the first-order automorphism $x\mapsto x+\varepsilon\cdot \ell(x)$ of $\bP^{14}$,
the family of cubic hypersurface we get is identically the Pfaffian,
whence the corresponding first-order deformation must be entirely contained in $G$, contradicting our analysis of their deformation in \autoref{54again}.

Consider the intermediate subvariety
\[
\widetilde{\cZ}\coloneqq Z(\tfrac{\partial F}{\partial x_i})_{i=0,\dots,14}
\]
whose fiber $\widetilde{\cZ}_\eta$ over the generic point $\eta$ coincides with $\Sing(\cY_\eta)$, while the fiber $\widetilde{\cZ}_0$ over $0$ coincides with $\Sing(\cY_0)=G$.
In particular, $\widetilde{\cZ}$ contains $\cX$.
Let $\cZ$ denote the schematic closure of $\widetilde{\cZ}_\eta$ in $\widetilde{\cZ}$ (namely, the union of the irreducible components of $\widetilde\cZ$ dominating $\Delta$); in particular, $\cZ$ is flat over $\Delta$.
Note that $\cZ$ contains $\cX$, since $\cX$ is the schematic closure of $\cX_\eta\subset \widetilde{\cZ}_\eta$. Therefore, $\cZ_0$ contains $\cX_0$ while being contained in $\widetilde{\cZ}_0 = G$ .

Now observe that for every linear relation
\[
\sum_{i=0}^{14}\ell_i^{(m)}\cdot \tfrac{\partial \pf}{\partial x_i}=0
\]
among the Plücker quadrics, the polynomial
\[
\sum_{i=0}^{14}\ell_i^{(m)}\cdot \tfrac{\partial F}{\partial x_i}
=\sum_{i=0}^{14}\ell_i^{(m)}\cdot \left(\tfrac{\partial \pf}{\partial x_i} + t\cdot \tfrac{\partial f_1}{\partial x_i} + O(t^2)\right)=t\left(\sum_{i=0}^{14}\ell_i^{(m)}\cdot\tfrac{\partial f_1}{\partial x_i} + O(t)\right)
\]
vanishes along $\cZ$, hence so does the polynomial
\[
\sum_{i=0}^{14}\ell_i^{(m)}\cdot\tfrac{\partial f_1}{\partial x_i} + O(t)=f_1^{(m)} + O(t).
\]
In particular, the cubic $f_1^{(m)}$ vanishes along $\cZ_0$.

By \autoref{prop:35cubics}, this implies that for the cubic hypersurface $G\cap Z(f_1)\in|\cO_G(3)|$, its singular locus $\Sing (G\cap Z(f_1))$ contains $\cZ_0$, which in turn contains $\cX_0=S^{[2]}$.
It follows from \autoref{prep-lemma-cubic}.(\ref{prep-lemma-cubic-1}) that $G\cap Z(f_1)$ is the first degeneracy locus $D_5(\varphi)$, whose singular locus is exactly $\cX_0$ by assumption.
Therefore, we have an equality $\cX_0=\cZ_0$.

Now we can conclude in the same way as in the proof \autoref{thm:coble-quartic}.
Looking at the Hilbert polynomials, for the generic point $\eta\in\Delta$ we have
\[
h_{\cX_\eta}\le h_{\cZ_\eta}= h_{\cZ_0}= h_{\cX_0},
\]
and therefore they must all be equal.
This proves that $\cX_\eta=\cZ_\eta=\Sing(\cY_\eta)$ as desired.
\end{proof}

We can also establish the minimal generators for the homogeneous ideal of $X\subset \bP^{14}$:

\begin{corollary}
\label{cor:square-6-homogeneous-ideal}
For a general $(X,H)\in\cM_6^{(1)}$, the homogeneous ideal of $X$ in $\bP^{14}$ is generated by the $15$ polar quadrics $\frac{\partial f}{\partial x_i}$ (with no linear syzygies) and $20$ cubics.
\end{corollary}
\begin{proof}
It was already established in \cite[Theorem~E.(2)]{projective-models} that $X$ is projectively normal and the homogeneous ideal of $X$ is generated by the $15$ polar quadrics plus at most $55$ cubics.
Using the Riemann--Roch formula, one immediately checks that the number of cubics required is $20$ plus the number of linear syzygies among $\frac{\partial f}{\partial x_i}$.
Therefore, it suffices to check that no such linear syzygy exists.

In view of \autoref{cor:no-syzygy}.(\ref{cor:no-syzygy-item}), we only need to check the injectivity of the natural map
\[
H^0(\bP^{14},\cT_{\bP^{14}}) \to H^0(X,\cN_{X/\bP}).
\]
Using the Euler sequence one can easily compute that $H^0(\bP^{14},\cT_{\bP^{14}}\otimes\cI_{X})=0$, hence $H^0(\bP^{14},\cT_{\bP^{14}})$ injects into $H^0(X,\cT_{\bP^{14}}|_X)$.
On the other hand, since $X$ is hyperkähler, we have $H^0(X,\cT_X)=H^0(X,\Omega_X)=0$ , hence $H^0(X,\cT_{\bP^{14}}|_X)$ injects into $H^0(X,\cN_{X/\bP})$, which concludes the proof.
\end{proof}

\section{Moduli spaces}
\label{sec:moduli_spaces}

In view of \autoref{thm:coble-quartic} and \autoref{thm:coble-cubic}, it is natural to consider the moduli spaces parametrizing the Coble type hypersurfaces, which should be birational to the moduli space $\cM_4^{(1)}$ and $\cM_6^{(1)}$ respectively.
In this section, we give a precise construction for these moduli spaces and study the geometry in relation to the moduli space of polarized hyperkähler fourfolds and the period domain.
We will focus mainly on the case $\cM_4^{(1)}$ and the corresponding Coble type quartics.
A similar analysis can of course also be carried out in the case of square~$6$.

\subsection{Moduli space and period map}
\label{subsec:moduli-intro}
We recall some basic properties of the moduli space $\cM_4^{(1)}$ and the period domain $\cP_4^{(1)}$, following the presentation in \cite[Section~3.9 and 3.10]{Debarre:survey}.

For a K3$^{[2]}$-type hyperkähler manifold $X$, there is a lattice structure
\[H^2(X,\bZ)\simeq\Lambda=\Lambda_{\KKK^{[2]}}\coloneqq U^{\oplus 3}\oplus E_8(-1)^{\oplus 2}\oplus\langle -2\rangle\]
given by the \emph{Beauville--Bogomolov--Fujiki} (BBF) form $q_X$.
We fix a primitive vector $h\in \Lambda$ of square $4$.%
\footnote{It will necessarily have divisibility $1$, and all such vectors lie in the same $O(\Lambda)$-orbit by the Eichler's criterion \cite[Theorem~2.9]{Debarre:survey}.}
For $(X,H)\in\cM_4^{(1)}$, any isometry $H^2(X,\bZ)\simeq\Lambda$ that maps the polarization $H$ to $h$ will map the period of $X$ into the space
\[
\Omega_{h}\coloneqq\setmid{x\in\bP(h^{\perp}\otimes\bC)}{ q(x)=0,\ q(x+\overline x)>0}.
\]
Taking the quotient by all isometries in $O(\Lambda)$ fixing $h$, we obtain the period domain
\[
\cP_4^{(1)}\coloneqq O(\Lambda, h)\backslash \Omega_h,
\]
which is an irreducible quasiprojective normal variety of dimension $20$.
One can then define an algebraic period map
\[\p\colon \cM_4^{(1)}\into \cP_4^{(1)},
\]
which is an open immersion by the global Torelli theorem \cite[Theorem~1.10]{Markman}. 

For every primitive sublattice $K\subset \Lambda$ of signature $(1,1)$ containing $h$, let $\cD_{4,K}^{(1)}\subset\cP_4^{(1)}$ be the image of the hyperplane $\bP(K^\perp_\bC)\subset\bP(h^{\perp}_\bC)$ by the quotient.
For each positive integer $d$, we define the \emph{Heegner divisor of discriminant} $d$ as the union
\[
\cD_{d}=\cD_{4,d}^{(1)}\coloneqq \bigcup_{K,\ \disc(K^\perp)=-d}\cD_{4,K}^{(1)}.
\]
It turns out that the divisor $\cD_{d}$ is non-empty when
\[
d\equiv 0,2,4,6,8,12\pmod {16},
\]
and is always irreducible \cite[Proposition~4.1]{Debarre-Macri}.

The complement of the image of the period map $\p$ is the union of the three Heegner divisors $\cD_4$, $\cD_{16}$, and $\cD_{20}$ \cite[Theorem~6.1]{Debarre-Macri}.
On the other hand, by the surjectivity of the period map for marked hyperkähler manifolds, for each period point in this locus we can still find some $(X, H)$, but the class $H$ will not be ample (only big and nef).
More precisely (see \cite[Section 3]{van-Geemen-Kapustka}), for general points on the divisor $\cD_4$, the linear system $|H|$ on $X$ is a \emph{Hilbert--Chow contraction}: these are therefore Hilbert squares $(S^{[2]}, L_2)$, where $(S,L)$ is a general quartic K3 surface;
for general points on $\cD_{16}$, the linear system $|H|$ is a \emph{Brill--Noether contraction}, with the class of the contracted divisor having square $-2$ and divisibility $1$;
for general points on $\cD_{20}$, the linear system $|H|$ is a small contraction contracting a Lagrangian plane \cite{Wierzba-Wisniewski}.

\subsection{Moduli space of Coble type quartics}
Recall that a given $(X,H)\in\cM_4^{(1)}$ has Hilbert polynomial $h_X(t)\coloneqq\chi(X,H^{\otimes t})=\binom{2t^2+3}{2}$. We denote by $\cH$ the irreducible component of the Hilbert scheme $\Hilb^{h_X}(\bP^9)$ whose general point is a hyperkähler fourfold of $\KKK^{[2]}$-type, embedded in $\bP^9$ with a complete linear system of square $4$.
As observed in \cite[Section~4.7]{projective-models}, $\cH$ is generically smooth of dimension $119$.

We have a rational dominant map $\varphi\colon\cH\dasharrow \cM_4^{(1)}$, as well as a map
$\coble\colon \cH\dasharrow |\cO_{\bP^9}(4)|$ ($\coble$ stands for ``Coble'') sending a general $X\in\cH$ to the $1$-dimensional subspace $H^0(\cI_{X/\bP^9}^2(4))\subset H^0(\cO_{\bP^9}(4))$.
By \autoref{thm:coble-quartic}, the map $\coble$ is birational onto its image, with inverse given by taking the singular locus of a quartic.
We denote by $\Sigma\subset |\cO_{\bP^9}(4)|$ the closure of the image of $\coble$, which is clearly invariant under the action of $\SL(10)$.
We also denote by $\Sigma^\circ$ the open subset where the singular locus of the corresponding quartic is a smooth hyperkähler fourfold $(X,H)\in \cM_4^{(1)}$.
\begin{lemma}
$\Sigma$ is not contained in the unstable locus for the action of $\SL(10)$.
Moreover, any $[f_4]\in\Sigma^\circ$ is semistable.
\end{lemma}
\begin{proof}
From \autoref{prep-lemma}.(\ref{prep-lemma-2}), we see that for a smooth quadric $q$, the double quadric $[q^2]$ lies in $\Sigma\subset |\cO_{\bP^9}(4)|$.
We claim that $[q^2]$ is semistable for the action of $\SL(10)$, which then implies that a general $[f_4]\in \Sigma$ must be semistable as well.
We follow the argument in \cite{Debarre-Han-OGrady-Voisin} and use a result of Luna: $V_{10}$ is a faithful irreducible representation for the stabilizer $\SO(10)$ of $[q^2]$, hence $\SO(10)$ has finite index in its normalizer in $\SL(10)$ \cite[Proposition~5.4]{Debarre-Han-OGrady-Voisin}, which shows that $[q^2]$ is \emph{polystable} (semistable with closed orbit in the semistable locus) \cite[Corollaire~3]{Luna} and therefore semistable \emph{a fortiori}.

For $[f_4]\in\Sigma^\circ$, we see that it lies in the regular locus of the rational map $\varphi\circ \coble^{-1}\colon\Sigma^\circ \to \cM_4^{(1)}$.
On the other hand, denoting by $\Sigma^{\circ,ss}\coloneqq \Sigma^\circ\cap \Sigma^{ss}$ the semistable locus, we may consider the GIT quotient $\Sigma^{\circ,ss}\to \Sigma^{\circ,ss}\git\SL(10)$ and use the fact that it is categorical to complete the diagram
\[
\begin{tikzcd}
\Sigma^{\circ,ss}\ar[r,hookrightarrow]\ar[d]& \Sigma^\circ\ar[d]\\
\Sigma^{\circ,ss}\git\SL(10)& \cM_4^{(1)}
\end{tikzcd}
\]
which, by the Zariski's main theorem, yields an immersion $\Sigma^{\circ,ss}\git\SL(10)\into\cM_4^{(1)}$.
Now let $D$ be a nonzero effective Cartier divisor in $\cM_4^{(1)}$ avoiding $(\varphi\circ \coble^{-1})([f_4])$ and intersecting the dense subset $\Sigma^{\circ,ss}\git\SL(10)$.
Then (the closure of) its preimage in $\Sigma$ is an $\SL(10)$-invariant divisor avoiding $[f_4]$.
Therefore $[f_4]$ is indeed semistable.
\end{proof}
Therefore, we may consider the GIT quotient
\begin{equation}
\Mcoble4\coloneqq\Sigma\git \SL(10),
\end{equation}
which is a nonempty irreducible projective variety and serves as a compactified moduli space for Coble type quartics.
In this way, $\coble$ descends to a birational map $\overline{\coble}\colon\cM_4^{(1)}\dasharrow\Mcoble4$. The situation can be summarized in the following commutative diagram:
\[
\begin{tikzcd}
\cH^{\dim 119}\ar[r,dashrightarrow,yshift=0.6ex,"\coble"]\ar[d,dashrightarrow]
& \Sigma^{\dim 119} \ar[l,dashrightarrow,yshift=-0.6ex,"\Sing"]\ar[r,hookrightarrow]\ar[d,dashrightarrow]
& {|\cO(4)|}^{\dim 714}\ar[d,dashrightarrow]\\
{\cM_4^{(1)}}^{\dim 20}\ar[r,dashrightarrow,yshift=0.6ex,"\overline{\coble}"]
& {\Mcoble4}^{\dim 20}\ar[l,dashrightarrow,yshift=-0.6ex,"\m"]\ar[r,hookrightarrow]
& {|\cO(4)|\git \SL(10)}^{\dim 615}
\end{tikzcd}
\]
where we denote by $\m\colon \Mcoble4 \dasharrow \cM_4^{(1)}$ the rational inverse of $\overline{\coble}$ ($\m$ stands for ``modular'').

We remark that the construction above works, \emph{mutatis mutandis}, in the case of Coble type cubics for $\cM_6^{(1)}$
\[
\Mcoble3\coloneqq \Sigma_{\mathrm{Coble}3}\git \SL(15)\subset {|\cO_{\bP^{14}}(3)|\git \SL(15)}^{\dim 455}.
\]
The only noteworthy change is the semistability of the Pfaffian cubic, which we spell out as a lemma.
\begin{lemma}
The Pfaffian cubic $[\pf]\in |\cO_{\bP^{14}}(3)|$ is polystable for the action of $\SL(15)$.
\end{lemma}
\begin{proof}
First, we claim that the stabilizer subgroup $\Stab(\pf)$ is $\SL(6)\subset \SL(15)$ via $\bw2 V_6$.
Clearly $\SL(6)$ is a subgroup of $\Stab(\pf)$; conversely, any element in $\SL(15)$ stabilizing $\pf$ induces an automorphism of its singular locus which is the Grassmannian $G=\Gr(2,6)$, hence must lie in $\SL(6)$.
Now since $\bw2 V_6$ is a faithful irreducible representation of $\Stab(\pf)=\SL(6)$, one can apply \cite[Proposition~5.4]{Debarre-Han-OGrady-Voisin} to obtain that $\SL(6)$ has finite index in its normalizer in $\SL(15)$, and then conclude using the result of Luna.
\end{proof}

Going back to the case of Coble type quartics for $\cM_4^{(1)}$,
by composing with the period map $\p\colon \cM_4^{(1)}\into \cP_4^{(1)}$,
and considering the Baily--Borel compactification $\overline{\cP_4^{(1)}}$ for the period domain, we get an extended period map $\tilde \p$ for Coble type quartics after resolving the indeterminacy:
\[
\begin{tikzcd}
\tilde{\Mcoble4}\ar[r,"\tilde\p"]\ar[d,"\varepsilon"] & \overline{\cP_4^{(1)}}\\
\Mcoble4\ar[r,dashrightarrow,"\p\circ \m"] & \cP_4^{(1)}\ar[u,hookrightarrow].
\end{tikzcd}
\]
\begin{definition}\label{def:HLS}
An irreducible divisor in $\cP_4^{(1)}$ (or $\cM_4^{(1)}$) is called \emph{Hassett--Looijenga--Shah} (HLS) if its closure in $\overline{\cP_4^{(1)}}$ is the image of an exceptional divisor of $\varepsilon$ by the extended period map $\tilde \p$.
\end{definition}
An equivalent formulation is the following: the modular map $\m$ is the rational inverse of the map $\overline{\coble}$, while the latter can be extended to a birational map $\cP_4^{(1)}\dasharrow \Mcoble4$, still denoted as $\overline{\coble}$ by abuse of notation.
Since $\cP_4^{(1)}$ is normal and in particular regular in codimension $1$, we can assume that $\overline{\coble}$ is defined in codimension $1$.
Then the HLS divisors are precisely those contracted by $\overline{\coble}$ to some image with higher codimension.

Using the analysis for $\cD_{12}$ in \autoref{subsec:genus-7} and the one for $\cD_4$ in \autoref{subsec:D4} below, we can deduce the following result.
\begin{theorem}
\label{thm:HLS}
The Heegner divisors $\cD_4$ and $\cD_{12}$ are HLS divisors.
\end{theorem}
\begin{proof}
A general element in $\cD_{12}$ corresponds to $(S^{[2]}, L_2-2\delta)$, where $(S,L)$ is a general K3 surface of genus $7$. Recall from \autoref{prep-lemma}.(\ref{prep-lemma-2}) that, for such an element, the non-reduced quartic $Z(q^2)$ is the unique quartic hypersurface singular along $S^{[2]}$.
Therefore the map $\overline{\coble}$ contracts the entire $\cD_{12}$ to the point $[q^2]\in \Mcoble4$.

Similarly, a general member in $\cD_{4}$ is a pair $(S^{[2]}, L_2)$ where $(S,L)$ is a general quartic K3 surface. The linear system $|L_2|$ defines the Hilbert--Chow contraction $S^{[2]}\to \Sym^2 S\subset\bP^9=\bP\Sym^2 V_4$, and by \autoref{cor:D4-coble-unique} below, the discriminant quartic in $\bP^9$ is the unique quartic hypersurface singular along $\Sym^2 S$.
Hence the map $\overline{\coble}$ contracts the entire $\cD_4$ to a single point $[f_{\disc}]$ in $\Mcoble4$.
\end{proof}

\subsection{Heegner divisor \texorpdfstring{$\cD_{4}$}{D4}}
\label{subsec:D4}
The Heegner divisor $\cD_4$ lies in the complement of $\cM_4^{(1)}$ in the period domain.
Its general member is $(S^{[2]},L_2)$, where $(S,L)$ is a general quartic K3 surface.
Recall that $L_2$ is not ample but only big and nef, and the morphism
\[
\varphi_{L_2}\colon S^{[2]}\to \bP^9=\bP H^0(S^{[2]},L_2)^\vee
\]
is the Hilbert--Chow contraction $S^{[2]}\to \Sym^2S$ followed by a closed immersion.

We analyze the geometry of $(S^{[2]}, L_2)$, which is very similar to the case already studied in \autoref{lem:D6-square6}.

\begin{lemma}
\label{cor:D4-coble-unique}
Consider $(S^{[2]},L_2)$ for $(S,L)$ a general quartic K3 surface in $\bP^3=\bP H^0(L)^\vee$.
\begin{enumerate}
\item\label{cor:D4-coble-unique1} The symmetric square $\Sym^2 S$ is projectively normal in $\bP^9$.
\item The cubic equations for $\Sym^2 S$ in $\bP^9$ are the same as those for $\Sym^2\bP^3$.
\item The discriminant quartic in $\bP^9$ is the unique quartic hypersurface singular along $\Sym^2 S$.
\end{enumerate}
\end{lemma}
\begin{proof}
The morphism $\varphi_{L_2}\colon S^{[2]}\to \bP^9=\bP(H^0(L_2)^\vee)=\bP(\Sym^2 H^0(L)^\vee)$ is the composition
\[
S^{[2]}\overset{\mathrm{HC}}{\to}\Sym^2 S\into \Sym^2 \bP^3\into \bP(\Sym^2 H^0(L)^\vee)=\bP^{9}.
\]
Factoring the restriction $H^0(\bP^{9}, \cO(d))\to H^0(\Sym^2S, \cO(d))$ through $\Sym^2 \bP^3$ we obtain
\[
m_d\colon\Sym^d\Sym^2 H^0(L)\to \Sym^2\Sym^d H^0(L) \xrightarrow{\Sym^2 r_d} \Sym^2 H^0(L^{\otimes d}).
\]
The first arrow is surjective, while the second arrow is the symmetric square of the restriction $r_d\colon \Sym^d H^0(L) \to H^0(L^{\otimes d})$ on $S\subset\bP^3$, which is an isomorphism for $d\le 3$ and surjective for $d\ge 4$.
It follows that $\Sym^2 S$ is projectively normal in $\bP^9$, and the cubic equations for $\Sym^2 S$ are the same as those for $\Sym^2\bP^3$.

Finally, for any quartic hypersurface singular along $\Sym^2 S$, its polar cubics will vanish along $\Sym^2 \bP^3$, hence it will necessarily be singular along $\Sym^2 \bP^3$ as well.
Since the discriminant quartic in $\bP^9$ is the unique quartic hypersurface singular along $\Sym^2 \bP^3$, this concludes the proof.
\end{proof}

\begin{remark}
    Note that \autoref{cor:D4-coble-unique}.(\ref{cor:D4-coble-unique1}) yields a simple proof to the fact that a general $(X,H)\in\cM_4^{(1)}$ is projectively normal, which was previously established in \cite[Section~4.7]{projective-models}.
\end{remark}

Similarly to the case of $\cD_{12}$, to recover the fourfold $\Sym^2S$ we will need not just the degenerate Coble type quartic, but also a tangential quartic section of the singular locus $\Sym^2 \bP^3$.
We quickly comment on its geometry.

The symmetric square $\Sym^2 \bP^3$ admits a resolution by $\bP_G\coloneqq\bP_G(\Sym^2 \cU)$ where $G\coloneqq \Gr(2, V_4)$ and $\cU$ is the tautological subbundle of rank $2$.
For a line $\ell$ in $\bP^3$, the fiber of the projection map $\pi\colon\bP_G\to G$ at $\ell$ is precisely the symmetric square $\Sym^2\ell\cong\bP^2$.
We denote by $\zeta$ the relative $\cO(1)$ of $\pi$.

Now let $(S,L)$ be a general quartic surface in $\bP^3=\bP(V_4)$ that contains no lines.
$S^{[2]}$ can be embedded in $\bP_G$ as a $6:1$ covering over $G$: fiberwise, for a general line $\ell$ in $\bP^3$ intersecting $S$ along $4$ distinguished points $\set{p_i}_{1\le i\le 4}$,
we obtain $6=\binom{4}{2}$ points in the fiber $\pi^{-1}(\ell)\cong\Sym^2\ell$.
Moreover, these $6$ points are the intersection of the four lines $L_i\coloneqq\setmid{\xi\subset\Sym^2\ell}{p_i\in \Supp\xi}$.
Globalizing this, we get a hypersurface $Y$ in $\bP_G$ of class $4\zeta$ that is singular along $S^{[2]}$, and indeed provides a quartic section after it is mapped to $\Sym^2 \bP^3$ via the linear system $|\cO(\zeta)|$.

\section{Speculations}
\label{sec:speculations}
\subsection{Global geometry of the moduli space}
We formulate some speculations on the moduli spaces $\Mcoble4$ and $\cM_4^{(1)}$, based on analogy with the cubic fourfolds \cite{Voisin:torelli,Hassett,Looijenga,Laza:moduli-via-period}, the double EPW sextics \cite{OGrady:taxonomy,OGrady:moduli-EPW}, and the Debarre--Voisin fourfolds \cite{Debarre-Han-OGrady-Voisin,Oberdieck,Song:special-DV}.

\begin{conjecture}
\label{conj:HLS}
\leavevmode
\begin{enumerate}
\item The Heegner divisors $\cD_2$, $\cD_6$, and $\cD_8$ are also HLS divisors.
\item\label{conj:HLS-2d} The Heegner divisors $\cD_{d}$ for $d\ge 16$ are not HLS.
\end{enumerate}
\end{conjecture}
The first part should follow once we examine the examples in \autoref{subsec:menagerie}.
For the second part, see \autoref{rmk:special-example} below.

\begin{conjecture}
\label{conj:smoothness}
Recall that $\Sigma^\circ\subset |\cO_{\bP^9}(4)|$ is the locus of Coble type quartics that are singular exactly along a \emph{smooth} hyperkähler fourfold $\in \cM_4^{(1)}$, and $\Sigma$ is its closure.
\begin{enumerate}
\item\label{conj:smoothness-divisorial} The complement $\Sigma\setminus \Sigma^\circ$ is a (reducible) divisor.
\item\label{conj:smoothness-two-comps} More precisely, the complement $\Sigma\setminus \Sigma^\circ$ consists of two irreducible divisors $\Sigma_{16}$ and $\Sigma_{20}$.
In other words, there are two $\SL(10)$ invariants $\Delta_{16}$ and $\Delta_{20}$ such that for $[f_4]\in \Sigma$, the singular locus of the quartic hypersurface $Z(f_4)$ is a smooth hyperkähler fourfold $\in\cM_4^{(1)}$ if and only if $\Delta_{16}(f_4)\ne 0$ and $\Delta_{20}(f_4)\ne 0$.
\end{enumerate}
\end{conjecture}
We note that these divisors are the analogue of singular cubic fourfolds (where a general member is a nodal cubic).
In this way, we get two irreducible divisors $\Mcoble{4,16}$ and $\Mcoble{4,20}$ in $\Mcoble4$ that dominate the divisors $\cD_{16}$ and $\cD_{20}$ respectively via the extended period map $\tilde\p$ (so these are indeed not HLS).
\begin{conjecture}
\label{conj:stablility}
Any $[f_4]\in\Sigma^\circ$ is stable for the action of $\SL(10)$.
In other words, we have an inclusion $\Sigma^\circ\subset \Sigma^s$, where $\Sigma^s$ is the stable locus.
\end{conjecture}
This would imply that the moduli space $\Mcoble4$ contains the following open loci
\[
\Mcoble4^\circ\coloneqq \Sigma^\circ / \SL(10) \subset \Mcoble4^s\coloneqq \Sigma^s / \SL(10),
\]
both being geometric quotients.
Then one can see immediately that the regular locus of $\m$ contains the open locus $\Mcoble4^\circ$ and induces an open immersion
\[
\m\colon \Mcoble4^\circ \into \cM_4^{(1)}
\]
by the Zariski's main theorem.
\begin{conjecture}
\label{conj:image}
Assuming \autoref{conj:stablility}, the complement of $\Mcoble4^\circ$ in $\cM_4^{(1)}$ is the union of the four HLS Heegner divisors $\cD_2$, $\cD_6$, $\cD_8$, and $\cD_{12}$.
(Recall that the complement of $\cM_4^{(1)}$ in $\cP_4^{(1)}$ is the union of three Heegner divisors $\cD_4$, $\cD_{16}$, and $\cD_{20}$.)
\end{conjecture}

Now we establish some partial results and implications.
\begin{proposition}
\autoref{conj:smoothness}.(\ref{conj:smoothness-divisorial}) implies \autoref{conj:stablility}.
\end{proposition}
\begin{proof}
Assuming \autoref{conj:smoothness}.(\ref{conj:smoothness-divisorial}), we will show that any $[f_4]\in \Sigma^\circ$ is stable, in other words, it admits a finite stabilizer subgroup $\Stab(f_4)$ and its affine orbit $O(f_4)$ in $H^0(\cO_{\bP^9}(4))$ is closed.

We first prove the first part, which holds unconditionally.
By assumption, the singular locus of the quartic hypersurface $Z(f_4)$ defines a hyperkähler fourfold $(X,H)\in \cM_4^{(1)}$.
Any non-trivial projective transformation stablizing the quartic will induce a non-trivial automorphism of $X$.
In other words, the natural homomorphism
\[
\Phi\colon \Stab(f_4) \to \Aut(X)
\]
has kernel $\setmid{\lambda \Id}{\lambda^{10}=1}$.
The automorphism group $\Aut(X)$ is discrete since $X$ is hyperkähler, thus so is $\Stab(f_4)$.
The latter, being an algebraic subgroup of $\SL(10)$, must necessarily be finite.

For the second part, assuming \autoref{conj:smoothness}.(\ref{conj:smoothness-divisorial}), then $\Sigma^\circ$ is the complement in $\Sigma$ of a (reducible) divisor given by an $\SL(10)$ invariant $\Delta$ (which, assuming \autoref{conj:smoothness}.(\ref{conj:smoothness-two-comps}), should be the product $\Delta_{16}\cdot \Delta_{20}$).
The existence of such a \emph{discriminant} allows one to conclude following the same argument as in the case of hypersurfaces \cite[Proposition~4.2]{Mumford-Fogarty-Kirwan}.
More concretely, suppose that the affine orbit $O(f_4)$ is not closed so its closure $\overline{O(f_4)}$ contains another orbit $O(f_4')$, then the latter must have higher codimension and therefore an infinite stabilizer, so it must be contained in the complement of $\Sigma^\circ$.
Thus we have $\Delta(f_4')=\Delta(f_4)=0$, a contradiction.
\end{proof}

\begin{remark}
\label{rmk:special-example}
We note that the smoothness criterion in \autoref{conj:smoothness} would also imply \autoref{conj:HLS}.(\ref{conj:HLS-2d}),
by following the same strategy as in \cite{Song:special-DV} and exhibiting some smooth examples with high Picard rank lying on these Heegner divisors; we already have some candidates, but the tricky part is to establish smoothness.
Hence \autoref{conj:smoothness} is really somewhat central, but very likely depends on the knowledge of an explicit parametrization.
This also highly motivates the study of the two Heegner divisors $\cD_{16}$ and $\cD_{20}$.

Alternatively, to show \autoref{conj:HLS}.(\ref{conj:HLS-2d}) one could attempt to compute the Noether--Lefschetz generating function, following the strategy in \cite{Oberdieck}.

On the other hand, \autoref{conj:image} would be the hardest one to establish, since it is a very strong statement on the global geometry of the GIT moduli space $\Mcoble4$.
\end{remark}

\subsection{Menagerie}
\label{subsec:menagerie}
We list some examples of degenerate Coble type quartics.
They all admit a singular locus with excessive dimension and have a Lie group symmetry.

The following two examples have already been studied.
\begin{itemize}
\item $\cD_{12}$: $f_4$ is the non-reduced quartic $q^2$, and its singular locus is the $8$-dimensional quadric.
It admits an $\SO(10)$-symmetry and is the unique $\SO(10)$-invariant quartic in $\Sym^4 V_{10}^\vee$, where $V_{10}$ is the standard representation.
\item $\cD_{4}$: $f_4$ is the discriminant quartic for symmetric $2$-tensors on a $\bP^3$ and is singular along the rank-$2$ locus which is a $\Sym^2\bP^3$.
It admits an $\SL(4)$-symmetry and is the unique $\SL(4)$-invariant quartic in $\Sym^4 V_{10}^\vee$, where $V_{10}=\Sym^2 V_4$.
\end{itemize}

Moreover, we have the following examples that we expect to be the degenerate Coble type quartic for the corresponding Heegner divisor (which would then be HLS following the same proof as \autoref{thm:HLS}).
\begin{itemize}
\item $\cD_{6}$: We expect an $\SO(5)$-symmetry.
We consider $\Sym^4 V_{10}^\vee$ where $V_{10}=\bw2 V_5$ and $V_5$ is the standard representation.
We obtain a pencil of $\SO(5)$-invariant quartics, and
there exists a distinguished member in the pencil whose singular locus contains $\Gr(2,V_5)$.
Recall that a general member in $\cD_6$ is given by $(S^{[2]},L_2-\delta)$ for a general K3 surface $(S,L)$ of genus $4$, mapped birationally onto a non-normal fourfold in $\Gr(2,V_5)$ \cite[Example~3.11]{Debarre:survey}.

\item $\cD_{8}$: We expect an $\SL(3)\times \SL(3)$-symmetry.
We consider $\Sym^4 V_{10}^\vee$ where $V_{10}=(V_3\otimes V_3)\oplus V_1$ and $V_3$ is the standard representation of $\SL(3)$.
We obtain a pencil of $\SL(3)\times \SL(3)$-invariant quartics.
One distinguished member in the pencil is reducible with singular locus containing a cone $C(\bP^2\times \bP^2)$.
By \cite[Theorem~1.1]{Iliev-Kapustka-Kapustka-Ranestad}, a general member $(X,H)\in\cD_8$ is hyperelliptic: $X$ is mapped two-to-one onto a quartic section of $C(\bP^2\times \bP^2)$, a so-called \emph{EPW quartic}.
One would expect that the Coble type quartic degenerates to the above one, while the second quartic is recovered by the tangential quartic.

\item $\cD_2$: We expect an $\SL(3)$-symmetry.
We consider $\Sym^4 V_{10}^\vee$ where $V_{10}=\Sym^3 V_3$ and $V_3$ is the standard representation of $\SL(3)$.
There is a unique $\SL(3)$-invariant quartic, the \emph{Aronhold quartic} $f_4$, a classical invariant for tenary cubic forms.
The singular locus of $f_4$ is a fivefold, containing a $\bP^2\times\bP^2$ embedded with bidegree $(2,1)$ of degree~$24$ (see \cite[Section~2.2]{Eisenbud-Harris}).
Since a general member $(X,H)$ of $\cD_2$ is latticed polarized with $\NS(X)=\bZ e\oplus \bZ f$ and Gram matrix $\begin{psmallmatrix}0&1\\1&0\end{psmallmatrix}$, where the polarization $H$ is $2e+f$,
we expect this $\bP^2\times\bP^2$ to be the hyperelliptic image of $X$.
\end{itemize}

Finally, we also briefly describe the picture for the Heegner divisors $\cD_{16}$ and $\cD_{20}$.

\begin{itemize}
\item For a general $(X,H)\in \cD_{16}$, recall that $H$ contracts a prime exceptional divisor $D$ with square $-2$ and divisibility $1$ to a quartic K3 surface $S\subset \bP^3$, giving a singular model $\overline{X}\subset \bP^9$.
We expect the scheme-theoretical singular locus of the corresponding Coble type quartic $Y$ to be $\overline{X}\cup\bP^3$.

Moreover, if one assumes that $X$ has Picard rank $2$ with $\NS(X)=\bZ H\oplus \bZ D$, then the contraction $D\to S$ induces a non-trivial Brauer class $\beta$ on $S$.

On the other hand, van Geemen--Kapustka \cite{van-Geemen-Kapustka} studied the geometry of the linear system $|H-D|$ that is ample of square $2$ and realizes $(X,H-D)$ as a double EPW sextic $X\xrightarrow{2:1}Y_6\subset\bP^5$, belonging to the Heegner divisor $\cD_{2,16}^{(1)}\subset \cM_2^{(1)}$.
Moreover, the divisor $D$ is mapped birationally onto a quadric section $Y_6\cap Q$.

We will investigate, in an upcoming work, how the Coble type quartic $Y$ provides a unified perspective relating these pictures, namely, the projective model of a twisted quartic K3 surface $(S,\beta)$ and a Noether--Lefschetz special EPW sextic $Y_6$ together with a distinguished quadric $Q$ in $\bP^5$.
\item For a general $(X,H)\in\cD_{20}$, recall that $H$ contracts a Lagrangian $\bP^2\subset X$ to a point $p$, giving a singular model $\overline{X}\subset\bP^9$.
Based on computer experiments, we expect the scheme-theoretical singular locus of the corresponding Coble type quartic $Y$ to be $\overline{X}$ with an embedded point at $p$.
\end{itemize}

\bibliography{refs}
\bibliographystyle{amsalpha}

\end{document}